\documentclass[letterpaper,11pt]{article}

\usepackage{setspace}
\usepackage[utf8]{inputenc}
\usepackage[margin=1in]{geometry}
\usepackage{array}
\usepackage{ stmaryrd,amsxtra,amstext,amscd,amsfonts,fancyhdr}
\usepackage{ amssymb }
\usepackage{amsmath}
\usepackage{amsthm}
\usepackage{enumerate}

\usepackage{tikz}
\usetikzlibrary{calc}

\newcommand\del[1]{}

\allowdisplaybreaks

\makeatletter
\@namedef{subjclassname@2010}{%
  \textup{2010} Mathematics Subject Classification}
\makeatother

\newtheorem{theorem}{Theorem}[section]
\newtheorem{corollary}[theorem]{Corollary}
\newtheorem{proposition}[theorem]{Proposition}

\theoremstyle{definition}
\newtheorem{definition}[theorem]{Definition}
\newtheorem{lemma}[theorem]{Lemma}

\newtheorem{example}[theorem]{Example}
\newtheorem{remark}[theorem]{Remark}

\numberwithin{equation}{section}

\newcommand\R{{\mathbb{R}}}
\newcommand\C{{\mathbb{C}}}

\newcommand\pd{{\partial}}
\newcommand\Z{{\mathbb{Z}}}

\renewcommand\L{{\mathcal{L}}}
\newcommand{\hk}{\mathbin{\! \hbox{\vrule height0.3pt width5pt depth 0.2pt \vrule height5pt width0.4pt depth 0.2pt}}}

\renewcommand\o{\overline }

\newcommand\Rho{\mathcal{P}}
\newcommand\g{\mathfrak{g}}
\newcommand\X{\mathfrak{X}}
\begin{document}
%\maketitle

\title{N$\ddot{\text{o}}$ether's Theorem}
\author{Jonathan Herman}
%\begin{titlepage}
\begin{center}
\setlength{\parindent}{0pt}
\setlength{\parskip}{17pt}

\vspace*{10pt}

{\Large \bf{Noether's Theorem in Multisymplectic Geometry}\rm\\ %TITLE LINE 2 IF NECESSARY \\ TITLE LINE 3 IF NECESSARY %\\ Thesis Template
\vspace{0.6cm}
\del{
\Large{by}
\vspace{-0.2cm}
}
\Large{Jonathan Herman}}
\\

{\it Department of Pure Mathematics, University of Waterloo}
\\

\tt{j3herman@uwaterloo.ca}
\vspace*{10pt}

\del{

{\small A research paper\\  %LARGE
\vspace{0.1cm}
presented to the University of Waterloo\\
\vspace{0.1cm}in fulfilment of the\\
\vspace{0.1cm}research paper requirement for the degree of\\
\vspace{0.1cm}Master of Mathematics\\
\vspace{0.1cm}in\\
\vspace{0.1cm}Pure Mathematics

\vspace*{\fill}

Waterloo, Ontario, Canada, 2014

\copyright \  Jonathan Herman, 2014

}
}
\end{center}

%\end{titlepage}

 \del{

\cleardoublepage

 \pagenumbering{roman}
 \setcounter{page}{2}
 
\vspace*{0.5in}
\noindent
{   \setlength{\parindent}{0pt}
   \setlength{\parskip}{24pt}
   \setlength{\textwidth}{7in}

   {\sffamily\bfseries \index{copyright!author's declaration}
   AUTHOR'S DECLARATION}

   I hereby declare that I am the sole author of this research paper.  This is a true
   copy of the research paper, including any required final revisions, as accepted by
   my examiners.

   I understand that my research paper may be made electronically available to the
   public. 
   
   \vspace*{1.0in}
   
Jonathan Herman
}
}

%%%%%%%%%%%%%%%%%%%%%%%%%%%%%%%%%%%%%%%%%%%%%%%%%%%%%%%%%%%%%%%%%%%

\begin{abstract}

We extend Noether's theorem to the setting of multisymplectic geometry by exhibiting a correspondence between conserved quantities and continuous symmetries on a multi-Hamiltonian system. We show that a homotopy co-momentum map interacts with this correspondence in a way analogous to the moment map in symplectic geometry. 
\del{
Our symmetries and conserved quantities will constitute an $L_\infty$-algebra and we demonstrate how a co-momentum map restricts to an $L_\infty$-morphism between the two. Under certain quotients, we will show that there is actually an isomorphism of graded Lie algebras between the symmetries and conserved quantities.
}

We apply our results to generalize the theory of the classical momentum and position functions from the phase space of a given physical system to the multisymplectic phase space. We also apply our results to manifolds with a torsion-free $G_2$ structure.

\end{abstract}

%%%%%%%ACKNOWLEDGEMENTS%%%%%%%%%%%%%%%%%%%%%%%%%%%%%%%%%%%%%%%%%%%%%%%%%%%
\del{
\onehalfspacing
\cleardoublepage
\section*{Acknowledgements}
First and foremost, a very sincere thank you goes to my supervisor Dr. Spiro Karigiannis. I am extremely grateful for his patience and the tremendous amount of time he spent teaching and helping me this summer. I would also like to thank Dr. Shengda Hu, my second reader, for the very useful comments and corrections that he provided. I need to acknowledge two of my good friends; Janis Lazovksis and Cameron Williams. Janis is a LaTeX machine, and I am very appreciative of the time he spent helping me this summer. Cam was a great support this year, always there to help me work through any problem. I also would like to thank my family for their constant love and interest, it means very much to me. Last, but not least (I'd say least goes to Cameron Williams), a thank you goes to my cousin Matt Rappoport, for without our discussions some of the contents in this paper would not exist.
}

%%%%%%%%%%%%% DEDICATION

%\doublespacing
\del{

\cleardoublepage
\vspace*{70pt}
\begin{center}
\itshape OPTIONAL DEDICATION CAN GO HERE.
\end{center}

}
%%%%%%%%%%%%%%%%%%%%%%%%%%%%%%%%%%%%%%%%%%%%%%%%%%%%%%%%%%%%%%

\tableofcontents

\pagenumbering{arabic}
\setcounter{page}{1}

\section{Introduction}
\del{
Noether's theorem says that for every symmetry on a classical Hamiltonian system, there is a corresponding conserved quantity. That is, the total energy of a conservative physical system in $\R^3$ is conserved because Hamilton's equations are invariant under time translation. Similarily, conservation of linear and angular momentum are consequences of the invariance of Hamilton's equations under space translations and rotations. In this paper, we replace the symplectic form with an arbitrary non-degenerate form (i.e. work on a multisymplectic manifold) and show how Noether's theorem generalizes. 
}

In this paper, we generalize concepts from Hamiltonian mechanics to multisymplectic geometry. In particular, we are interested in the relationship between continuous symmetries and conserved quantities on multi-Hamiltonian systems. Recall that  a symmetry on a symplectic Hamiltonian system $(M,\omega,H)$ is a symplectic vector field whose flow preserves the Hamiltonian. A conserved quantity is a function that is constant along the motions of the given physical system.  
\vspace{0.3cm}

Noether's theorem sets up a correspondence between conserved quantities and continuous symmetries. For example, under a conservative system in $\R^3$, the flows of each vector field coming from the Lie algebra action of $\mathfrak{so}(3)$ all leave the total energy, $H$, invariant. Thus, each of the infinitesimal generators from $\mathfrak{so}(3)$ are continuous symmetries.  An example of a conserved quantity would be the angular momentum of a particle moving in this physical system. The angular momentum (in a specified direction) is a function on $T^\ast\R^3$ that is constant on the motion. Noether's theorem sees conservation of angular momentum as a consequence of the rotational symmetry. 
\vspace{0.3cm}

In a multi-Hamiltonian system $(M,\omega,H)$, with $\omega\in\Omega^{n+1}(M)$ and $n\geq 1$, the Hamiltonian $H$ is now a `Hamiltonian form' of degree $n-1$. Here Hamiltonian means that there exists a unique vector field $X_H$ satisfying $X_H\hk\omega=-dH$. The flow of $X_H$ gives the dynamics of the given physical system. \del{  For example, if $M=T^\ast\R^3=\R^6$ and $\omega=\omega_0^2$, then one could take $H$ to be the electric two form $E$. It's true that a Hamiltonian vector field for $X_E$ is given by the motion in an electric field (TO DO). }
\vspace{0.3cm}

A notion of conserved quantity on a multisymplectic manifold was given in \cite{cq}. They defined three types; a differential form $\alpha$ is called a local, global, or strict conserved quantity if $\L_{X_H}\alpha$ is closed, exact or zero respectively. In our work we modify their definition by requiring that a conserved quantity $\alpha$ is also Hamiltonian, meaning that $X_\alpha\hk \omega=-d\alpha$ for some multivector field $X_\alpha$.  By adding in this requirement, we are then able to study how the extended `Poisson' bracket, $\{\cdot,\cdot\}$, from \cite{dropbox} interacts with the conserved quantities.  \del{The set of local, global and strict conserved quantities on a multi Hamiltonian system $(M,\omega,H)$ is denoted by $\mathcal{C}_{\text{loc}}(X_H)$, $\mathcal{C}(X_H)$ and $\mathcal{C}_{\text{str}}(X_H)$ respectively,}
\vspace{0.3cm}

We find that, analogous to the case of Hamiltonian mechanics, the Poisson bracket of two conserved quantities is always strictly conserved.
That is, \begin{proposition} Let $\alpha$ and $\beta$ be two (local, global or strict) conserved quantities on a multi-Hamiltonian system $(M,\omega, H)$. Then $\{\alpha,\beta\}$ is strictly conserved, meaning $\L_{X_H}\{\alpha,\beta\}=0$.
\end{proposition} From this proposition we will show that the conserved quantities, modulo closed forms, constitute a graded Lie algebra. We will also show that when restricted to a certain subspace, namely the Lie $n$-algebra of observables (see definition \ref{Lie n observables}), the conserved quantities form an $L_\infty$-algebra.
\vspace{0.3cm}

Similarly, we will see that our continuous symmetries also generate a graded Lie algebra. As an extension from Hamiltonian mechanics, we define a symmetry to be a Hamiltonian multivector field with respect to which the Lie derivative of the Hamiltonian has a specific form. Just as for the conserved quantities, we have three types of continuous symmetry. Namely, a multivector field $X$ is a local, global, or strict symmetry on $(M,\omega,H)$ if $\L_X\omega=0$ and $\L_XH$ is closed, exact, or zero respectively. A generalization from Hamiltonian mechanics is \begin{proposition}Given any two (local, global, strict) continuous symmetries $X$ and $Y$, their Schouten bracket $[X,Y]$ is a continuous symmetry of the same type.\end{proposition} From this proposition we will see that the continuous symmetries, modulo elements in the kernel of $\omega$, form a graded Lie algebra. 
\vspace{0.3cm}

Our first generalization of Noether's theorem says that there is a correspondence between these notions of symmetry and conserved quantity on a multisymplectic manifold.

\begin{theorem}
If $\alpha$ is a (local or global) conserved quantity, then every corresponding Hamiltonian multivector field $X_\alpha$ is a (local or global) continuous symmetry. Conversely, if $X$ is a (local or global) continuous symmetry, then every corresponding Hamiltonian form is a (local or global) conserved quantity.
\end{theorem}

As in symplectic geometry, this correspondence is not one-to-one. Indeed, for a Hamiltonian form, any two of its corresponding Hamiltonian multivector fields differ by an element in the kernel of $\omega$. Conversely, any two Hamiltonian forms corresponding to a Hamiltonian multivector field differ by a closed form:

Let  \[\Omega_{\text{Ham}}(M)=\{\alpha\in\Omega^\bullet(M); d\alpha= X\hk\omega \text{ for some } X\in\Gamma(\Lambda^\bullet(TM))\}\] denote the graded vector space of Hamiltonian forms, and let $\widetilde\Omega_{\text{Ham}}(M)$ denote the quotient of $\Omega_{\text{Ham}}(M)$ by closed forms. Similarily, we let \[\X_{\text{Ham}}(M)=\{X\in\Gamma(\Lambda^\bullet(TM)); X\hk \omega \text{ is exact}\}\]  denote the graded vector space of multi Hamiltonian vector fields and $\widetilde\X_{\text{Ham}}(M)$ denote the quotient of $\X_{\text{Ham}}(M)$ by elements in the kernel of $\omega$. We will then show, as in \cite{dropbox}, that $\{\cdot,\cdot\}$ descends to a well defined graded Poisson bracket on $\widetilde\Omega_{\text{Ham}}(M)$. Then we show that

\begin{theorem}
There is a natural isomorphism of graded Lie algebras between $(\widetilde\X_{\text{Ham}}(M),[\cdot,\cdot])$ and $(\widetilde\Omega_{\text{Ham}}(M),\{\cdot,\cdot\})$.
\end{theorem}

As a consequence of this theorem, we will then show that our symmetries and conserved quantities, after appropriate quotients, are in one-to-one correspondence. In particular, we let $\mathcal{C}_\text{loc}(X_H)$, $\mathcal{C}(X_H)$, $\mathcal{C}_\text{str}(X_H)$ denote the spaces of local, global, and strict conserved quantities respectively, and  $\widetilde{\mathcal{C}}_\text{loc}(X_H),\widetilde{\mathcal{C}}(X_H)$ and $\widetilde{\mathcal{C}}_{\text{str}}(X_H)$ their quotients by closed forms. Similarily, we let $\mathcal{S}_\text{loc}(H)$, $\mathcal{S}(H)$ and $\mathcal{S}_\text{str}(H)$ denote the space of local, global, and strict continuous symmetries respectively, and  $\widetilde{\mathcal{S}}_\text{loc}(H),\widetilde{\mathcal{S}}(H)$ and $\widetilde{\mathcal{S}}_\text{str}(H)$ their quotient by elements in the kernel of $\omega$. We obtain:
\begin{theorem}
There is an isomorphism of graded Lie algebras from $(\widetilde{\mathcal{S}}(H),[\cdot,\cdot])$ to $(\widetilde{\mathcal{C}}(X_H),\{\cdot,\cdot\})$ and from  $(\widetilde{\mathcal{S}}_{\text{loc}}(H),[\cdot,\cdot])$ to $(\widetilde{\mathcal{C}}_{\text{loc}}(X_H),\{\cdot,\cdot\})$. Moreover, there is an injective graded Lie algebra homomorphism from $(\widetilde{\mathcal{S}}_{\text{str}}(H),[\cdot,\cdot])$ to $(\widetilde{\mathcal{C}}(X_H),\{\cdot,\cdot\})$ and from   $(\widetilde{\mathcal{C}}_{\text{str}}(X_H),\{\cdot,\cdot\})$ to $(\widetilde{\mathcal{S}}(H),[\cdot,\cdot])$.
\end{theorem}

Furthermore, we will show that under certain assumptions for a group action on $M$, a homotopy co-momentum map $(f)$ (see definition \ref{hmm}) gives rise to a whole family of conserved quantities and continuous symmetries. A group action on a multi-Hamiltonian system $(M,\omega,H)$ is called locally, globally, or strictly $H$ preserving if the Lie derivative of $H$ under each infinitesimal generator from $\g$ is closed, exact, or zero respectively. 

In \cite{cq} it was shown that if the group locally or globally preserves $H$, then for any $p\in\Rho_{\g,k}$, the $k$-th Lie kernel (see definition \ref{Lie kernel}), $f_k(p)$ is locally conserved, and if the group strictly preserves $H$ then $f_k(p)$ is globally conserved. We add to this result by showing that under the above assumptions $V_p$ is a local or global continuous symmetry. In particular, let $S_k=\{V_p; p\in\Rho_{\g,k}\}$ denote the infinitesimal generators coming from the Lie kernel. Then $S=\oplus S_k$ is a differential graded Lie algebra. Let $C_k=\{f_k(p) ; p\in\Rho_{\g,k}\}$ denote the image of the moment map. We set $C=\oplus C_k$ and show that $C\cap L_\infty(M,\omega)$ is an $L_\infty$-subalgebra of $L_\infty(M,\omega)$, the Lie $n$-algebra of observables. We then obtain

\begin{theorem}
For a (locally, globally or strictly) preserving $H$-action, a homotopy co-momentum map induces an $L_\infty$-morphism from $S$ to $C\cap L_\infty(M,\omega)$.
\end{theorem}

We finish by giving two applications of our results:

Recall that in symplectic geometry, any vector field on the given manifold induces a vector field on the phase space, called the corresponding classical momentum function. Moreover, any function on the base manifold pulls back to a function on the phase space giving the so called classical position functions. These momentum and position functions on the phase space satisfy specific commutation relations that form the bridging gap from classical to quantum mechanics (see chapter 5.4 of \cite{Marsden}). As an application of our results we will replace vector fields and functions with arbitrary multivector fields and differential forms to obtain `position and momentum forms'. We will show that these position and momentum satisfy higher order bracket relations generalizing the bracket relations in symplectic geometry.
 
 Lastly, we finish by applying our work to manifolds with $G_2$-structure. In particular, we extend Example 6.7 of \cite{ms} by obtaining a homotopy co-momentum map for a $T^2$ action on a torsion-free $G_2$-manifold.
\del{

This issue of only obtaining locally defined conserved quantities from a continuous symmetries is also present in Hamiltonian mechanics. It is solved if one only considers the symmetries and conserved quantities coming from a moment map. We find the same thing happens in multisymplectic when one considers the symmetries and conserved quantities coming from a multi moment map. Moreover, we will have to restrict the domain of the moment map to the Lie kernel. In \cite{cq} it was shown that if a group acts on a multi Hamiltonian system $(M,\omega,H)$  and locally (or globally) preserves $H$ then for a homotopy co-momentum map $f_k:\Lambda^k\to\Omega^{n-k}(M)$ we have that $f_k(p)$ is a local conserved quantity for each $p\in\Rho_{g,k}$.

Considering homotopy moment maps restricted to the Lie kernel we obtain

\begin{proposition}
A homotopy moment map restricts to an $L_\infty$-morphism from $(\Rho_\g,[\cdot,\cdot],\del)$ to $\widetilde L_\infty(M,\omega)$.
\end{proposition}

Next we consider how the algebraic structures of conserved quantities and continuous symmeries are related. 
We let $C$ denote the symmetries coming from the moment map and $S$ the conserved quantities.
\begin{proposition}
A homotopy moment map gives an isomorphism of graded Lie algebras from $C(X_H)\cap \widetilde C$ to $S(H)\cap \widetilde S$.
\end{proposition}

This proposition is actually a consequence of something more general. Namely,

\begin{theorem}
There is a graded Lie algebra isomorphism from $C(X_H)\cap\Omega_{\text{Ham}}(M)$ to $S(H)\cap\X_{\text{Ham}}(M)$.
\end{theorem}

In symplectic geometry, Noether's theorem says that the map $V_\xi\mapsto\mu(\xi)$ gives a one-to-one correspondence between symmetries and conserved quantities. Here $\mu(\xi)$ is the conserved quantity coming from the symmetry $V_\xi$. Hence we mimic this correspondence to say that a a symmetry is a vector field $X$ satisyftynig $X\hk\omega$ is exact. 

Since a 
By adding in this requirement, we can use the 'generalized Poisson bracket' defined in \ref{dropbox} and say more about the conserved quantity.

In symplectic geometry the conserved quantities form a Lie subalgebra of $(C^\infty(M),\{\cdot,\cdot\})$. The continuous symmetries form a Lie subalgebra of $(\Gamma(TM),[\cdot,\cdot])$. Noether's theorem says that  and we study how this generalizes to multisymplectic geometry.

says that for a differential form $\alpha$ is conserved was a different

In this paper we use multisymplectic geometry  to generalize notions of symmetry and conserved quantity  Hamiltonian mechanics as extends concepts from symplectic geometry 
To formulate Noether's theorem, we use the notion of  conserved quantity on a multisymplectic manifold defined in \cite{cq}. Intuitively, 

Moreover, in \cite{dropbox} a generalization of the Poisson bracket to multisymplectic systems was introduced, and we show how this bracket interacts with the aforementioned conserved quantities. 

We define our continuous symmetries to be multivector field. In order to setup a corrsepondence between  Moreover, in \cite{cq} a notion of conserved quantity on a multisymplectic manifold was defined. 

The work in this paper connects several results given in \cite{cq}, \cite{ms}, \cite{dropbox}.

In \cite{cq} a notion of conserved quantity on a multyismplectic manifold was introduced, generalizing conservation laws from Hamiltonian mechanics. In \cite{ms}, Madsen and Swann explained how one could define a moment map on a multi symplectic manifold so that its image was contained in a finite dimensional vector space, analogous to the scenario in symplectic geometry. In \cite{dropbox} they give a notion of Hamiltonian structures on multisymplectic manifolds.

In this paper, we aim to give a notion of symmetry on a multiysmplectic manifold and set up a correspondence generalizing Noethers theorem in Hamiltonian mechanics.  We work with a fixed multi-Hamiltonian system $(M,\omega,H)$ where $\omega$ is $n$-plectic and $H\in\Omega^{n-1}_{\text{Ham}}(M)$ is Hamiltonian. We modify the definition of a conserved quantity given in \cite{cq}, by adding the requirement that a conserved quantity must be Hamiltonian. This allows us to make use of some of the structures developed in \cite{dropbox}. In this paper we build on these structures to give a notion of symmetry on a multisymplectic manifold. In particular, we will build on these structures to give a notion of symmetry on a multisymplectic manifold. Moreover, we show that these generalized Hamiltonian structures interact nicely with a homotopy co-moment  map. In particular,

\begin{proposition}
$\{\alpha,\beta\}$ is strictly conserved.
\end{proposition}

We then adopt  ideas from \cite{ms} by replacing the Chevalley-Eilenberg complex with a sub complex containing only Lie kernels. By doing this, we make a connection with symplectic geometry, that the image of each element in the Lie algebra under a moment map is a conserved quantity. The fact that each element of the Lie kernel gets mapped to a conserved quantity is one of the main ideas in \cite{cq}. In this paper we show that the infinitesimal generator of each element in the Lie kernel is a symmetry and its corresponding conserved quantity is that obtained in \cite{cq}.

That is, we will investigate how a momentum map converts symmetries to conserved quantities and vice versa.  Similar to \cite{cq}, we will have three types of symmetries, and will setup the correspondence in each of the three cases.

\begin{proposition}
Fix a multi-Hamiltonian system $(M,\omega,H)$ where $\omega\in\Omega^{n+1}(M)$ and $H\in\Omega^{n-1}_{\text{Ham}}(M)$. For an element $p$ in the Lie kernel $\Rho_{\g,k}$ we have that its infinitesimal generator is a global symmetry. Its corresponding conserved quantity is $f_k(p)$, where $(f)$ is a homotopy co-momentum map.
\end{proposition}

We get the following generalization of Noether's theorem.

\begin{theorem}
If $\alpha\in\Omega^{n-k}_{\text{Ham}}(M)$ is a (local, global) conserved quantity then $X_\alpha$ is a (local, global) continuous symmetry, for any Hamiltonian vector field $X_\alpha$ of $\alpha$. Conversely, if $A\in\Gamma(\Lambda^k(TM))$ is a continuous symmetry, then locally $A=X_\alpha$ for some $(n-k)$-form $\alpha$. If $A$ is a (local, global) continuous symmetry, then this form $\alpha$ is a (local, global) conserved quantity, locally.
\end{theorem}

We also show that a moment map is an $L_\infty$-morphism between various different spaces.

\begin{theorem}
A momentum map restricts to an $L_\infty$-morphism between $(\Rho_\g,\partial,[\cdot,\cdot])$ and $(\widetilde L, \{l_k\})$.

\end{theorem}

\begin{theorem}
A momentum map is an $L_\infty$-algebra morphism between $C(X_H)\cap \widetilde L$ and $S(H)\cap \mathfrak{X}$ (also local and strict).
\end{theorem}

We also exhibit how a moment map reduces to an Lie algebra isomorphism when restricted in a certain cohomological sense.

(This addresses two open problems proposed in \cite{questions}, namely the first and third bullet points of section 13.  Actually, the results concerning the 3rd bullet are only referenced here, the details are in other notes.)
}

\section{Differential Graded Lie Algebras and $L_\infty$-algebras}

We start by recalling some basic properties of differential graded Lie algebras.
\subsection{Differential Graded Lie Algebras}

We first need to recall the definition of a differential graded Lie algebra and a Gerstenhaber algebra.

\begin{definition}
A differential graded Lie algebra is a $\Z$-graded vector space $L=\oplus_{i\in\Z} L_i$ together with a bracket $[\cdot,\cdot]:L_i\otimes L_j\to L_{i+j}$ and a differential $d:L_i\to L_{i-1}$. The bilinear map $[\cdot,\cdot]$ is graded skew symmetric:  \[[x,y]=-(-1)^{|x||y|}[y,x],\] and satisfies the graded Jacobi identity: \[(-1)^{|x||z|}[x,[y,z]]+(-1)^{|y||x|}[y,[z,x]]+(-1)^{|z||y|}[z,[x,y]]=0.\] Lastly, the differential and bilinear map satisfiy the graded Leibniz rule: \[d[x,y]=[dx,y]+(-1)^{|x|}[x,dy].\] Here we have let $x,y,$ and $z$ be arbitrary homogeneous elements in $L$ of degrees $|x|,|y|$ and $|z|$ respectively.
\end{definition}

\begin{definition}
A Gerstenhaber algebra is a $\Z$-graded commutative algebra $A=\oplus_{i\in\Z}A_i$ with a bilinear map $[\cdot,\cdot]:A\otimes A \to A$ satisfying the following properties: 

\begin{itemize}

\item $|[a,b]|=|a|+|b|-1$ (the bilinear map has degree $-1$),

\item $[a,bc]=[a,b]c+(-1)^{(|a|-1)|b|}b[a,c]$ (the bilinear map satisfies the Poisson identity),

\item $[a,b]=-(-1)^{(|a|-1)(|b|-1)}[b,a]$ (the bilinear map is antisymmetric),
\end{itemize} and lastly, the bilinear map satisfies the Jacobi identity: 

\begin{itemize}
\item $(-1)^{(|a|-1)(|c|-1)}[a,[b,c]]+(-1)^{(|b|-1)(|a|-1)}[b,[c,a]]+(-1)^{(|c|-1)(|b|-1)}[c,[a,b]]=0$.
 \end{itemize}
 Here we have let $|a|$ denote the order of $a\in A$, and $ab$ the product of $a$ and $b$ in $A$.
\end{definition}
Let $(V,[\cdot,\cdot])$ be a Lie algebra. The Schouten bracket turns $\Lambda^\bullet V$, the exterior algebra of $V$, into a Gerstenhaber algebra. We quickly recall some properties of the Schouten bracket. A more detailed discussion can be found in \cite{marle}.

\begin{proposition}
The Schouten bracket is the unique bilinear map $[\cdot,\cdot]:\Lambda^\bullet V\times\Lambda^\bullet V \to\Lambda^\bullet V$ satisfying the following properties: 
\begin{itemize}
\item If $\deg X = k$ and $\deg Y= l$ then $\deg([X,Y])=k+l-1$.
\item $[X,Y]=-(-1)^{(k+1)(l+1)}[Y,X]$.
\item It coincides with the Lie bracket on $V$.
\item For $X\in\Gamma(\Lambda^k(V))$, $Y\in\Gamma(\Lambda^l(V))$ and $Z\in\Gamma(\Lambda^m(V))$, \[[X,Y\wedge Z]=[X,Y]\wedge Z+(-1)^{(k-1)l}Y\wedge[X,Z].\]
\item It satisfies the graded Jacobi identity: For $X,Y$ and $Z$ of degree $k,l$ and $m$ respectively, \[\sum_{\mathrm{cyclic}}(-1)^{(k-1)(m-1)}[X,[Y,Z]] = 0.\]
\end{itemize}
\end{proposition}

\begin{proof}
This is Proposition A.1 of \cite{poisson}.
\end{proof}

On decomposable multivectors $X=X_1\wedge\cdots \wedge X_k \in\Lambda^k V$ and $Y=Y_1\wedge\cdots \wedge Y_l\in\Lambda^lV$ the Schouten bracket is given by \[ [X,Y]:=\sum_{i=1}^k\sum_{j=1}^l(-1)^{i+j}[X_i,Y_j]X_1\wedge\cdots\wedge \widehat X_i\wedge\cdots\wedge X_k\wedge Y_1\wedge\cdots\wedge\widehat Y_j\wedge\cdots\wedge Y_l.\]

We now let $M$ be a manifold and consider the Gerstenhaber algebra $(\Gamma(\Lambda^\bullet(TM)),\wedge,[\cdot,\cdot])$. 

\begin{definition}
For a decomposable multivector field $X=X_1\wedge\cdots\wedge X_k$ in $\Gamma(\Lambda^k(TM))$ and a differential form $\tau$, we define the contraction of $\tau$ by $X$ to be \[X\hk\tau:=X_k\hk\cdots\hk X_1\hk\tau,\] and extend by linearity to all multivector fields. We define the Lie derivative of $\tau$ in the direction of $X$ to be \begin{equation}\label{Lie}\L_X\tau:=d(X\hk\tau)-(-1)^{k}X\hk d\tau.\end{equation} Note that this is the usual Lie derivative when $k=1$.
\end{definition}Throughout the paper we will make extensive use of the following propositions.

\begin{proposition}
\label{properties}
Let $X\in\Gamma(\Lambda^k(TM))$ and $Y\in\Gamma(\Lambda^l(TM))$ be arbitrary. For a differential form $\tau$, the following hold:

\begin{equation}\label{dL}d\L_X\tau=(-1)^{k+1}\L_Xd\tau\end{equation}

\begin{equation}\label{bracket hook}[X,Y]\hk\tau=(-1)^{(k+1)l}\L_X(Y\hk\tau)-Y\hk(\L_X\tau)\end{equation}

\begin{equation}\label{L bracket}\L_{[X,Y]}\tau=(-1)^{(k+1)(l+1)}\L_X\L_Y\tau-\L_Y\L_X\tau\end{equation}

\begin{equation}\L_{X\wedge Y}\tau=(-1)^lY\hk(\L_X\tau)+\L_Y(X\hk\tau)\end{equation}

\end{proposition}

\begin{proof}
This is Proposition A.3 of \cite{poisson}.
\end{proof}

Another formula for the interior product by the Schouten bracket is given by the next proposition.
\begin{proposition}
\label{interior}For $X\in\Gamma(\Lambda^k(TM))$ and $Y\in\Gamma(\Lambda^l(TM))$ we have that interior product with their Schouten bracket satisfies
 \[i[X,Y]=[-[i(Y),d],i(X)],\] where the bracket on the right hand side is the graded commutator. Written out fully, this says that for an arbitrary form $\tau$,
 \begin{equation}\label{interior equation}[X,Y]\hk\tau=-Y\hk d(X\hk\tau)+(-1)^ld(Y\hk X\hk\tau)+(-1)^{kl+k}X\hk Y\hk d\tau-(-1)^{kl+k+l}X\hk d(Y\hk\tau)\end{equation}
\end{proposition}

\begin{proof}
This is Proposition 4.1 of \cite{marle}. It can also be derived directly from equations (\ref{Lie}) and (\ref{bracket hook}). We state it as a separate proposition because equation (\ref{interior equation}) will be used frequently in the rest of the paper.
\end{proof}

Next we recall the Chevalley-Eilenberg complex. We start with a Lie algebra $\g$ and its exterior algebra $\Lambda^\bullet\g$. The Gerstehaber algebra $(\Lambda^\bullet\g,\wedge,[\cdot,\cdot])$ is turned into a differential algebra by the following differential. 

\begin{definition} For a Lie algebra $\g$, consider the differential 

\[\partial_k:\Lambda^k\g\to\Lambda^{k-1}\g, \ \ \ \ \ \xi_1\wedge\cdots\wedge\xi_k\mapsto\sum_{1\leq i<j\leq k}(-1)^{i+j}[\xi_i,\xi_j]\wedge\xi_1\wedge\cdots\wedge\widehat\xi_i\wedge\cdots\wedge\widehat\xi_j\wedge\cdots\wedge\xi_k\] for $k\geq 1$, and extend by linearity to non-decomposables. Define $\Lambda^{-1}\g=\{0\}$ and $\partial_0$ to be the zero map. It follows from the graded Jacobi identity that $\partial^2=0$.
The differential Gerstenhaber algebra $(\Lambda^k\g,\wedge,\partial,[\cdot, \cdot])$ is called the Chevalley-Eilenberg complex.\end{definition}

\begin{definition}\label{Lie kernel}
We follow the terminology and notation of \cite{ms} and call $\Rho_{\g,k}=\ker \partial_k$ the $k$-th Lie kernel, which is a vector subspace of $\Lambda^k\g$. Let $\Rho_\g$ denote the direct sum of all the Lie kernels: \[\Rho_\g=\oplus_{k=0}^{\dim(\g)}\Rho_{\g,k}.\] Note that if the group is abelian then $\Rho_{\g,k}=\Lambda^k\g$.
\end{definition}

A straightforward computation gives the following lemma.
\begin{lemma}\label{lemma formula}
For arbitrary $p\in\Lambda^k\g$ and $q\in\Lambda^l\g$ we have that 
\[\partial(p\wedge q)=\partial(p)\wedge q+(-1)^kp\wedge\partial(q)+(-1)^k[p,q].\]
\end{lemma}

From this lemma we get the following.
\begin{proposition} 
\label{Schouten is closed}
We have $(\Rho_{\g},\partial,[\cdot,\cdot])$ is a differential graded subalgebra of the Chevalley-Eilenberg complex, with $\partial=0$.
\end{proposition}

\begin{proof}
The only nontrivial thing we need to show is that the Schouten bracket preserves the Lie kernel. While this follows immediately from the fact that $\partial$ is a graded derivation of the Schouten bracket, we can also show it using Lemma \ref{lemma formula}. Indeed, for $p\in\Rho_{\g,l}$ and $q\in\Rho_{\g,l}$ we have that \begin{align*}\partial(p\wedge q)&=\partial(p)\wedge q +(-1)^k p\wedge\partial(q)+[p,q]\\
&=[p,q].\end{align*}Hence, $[p,q]$ is exact and therefore closed.
\end{proof}

We now define a differential graded Lie algebra consisting of multivector fields.

Let $G$ be a connected Lie group acting on a manifold $M$. For $\xi\in\g$ let $V_\xi\in\Gamma(TM)$ denote the infinitesimal generator of the induced action on $M$ by the one-parameter subgroup of $G$ generated by $\xi$.  For decomposable $p=\xi_1\wedge\cdots\wedge\xi_k$ in $\Lambda^k\g$ we introduce the notation $V_p:=V_{\xi_1}\wedge\cdots\wedge V_{\xi_k}$ for the infinitesimal generator of $p$.  Let \begin{equation}\label{S_k}S_k=\{V_p \ ;\  p\in\Rho_{\g,k}\}\end{equation} and set \begin{equation}S=\oplus_{k=0}^{\dim(\g)} S_k.\end{equation} 

\begin{proposition}
\label{infinitesimal generator of Schouten}
We have that $(S,[\cdot,\cdot])$ is a graded Lie algebra. Moreover, for $p\in\Lambda^k\g$, we have that $\partial V_p=-V_{\partial p}$. Note that we have abused notation and let $\partial$ denote the Chevalley-Eilenberg differentials for both the Lie algebras $(\Gamma(TM),[\cdot,\cdot])$ and $(\g,[\cdot,\cdot])$.\end{proposition}

\begin{proof}
We first show that $V_{[p,q]}=-[V_p,V_q]$. Let $p=\xi_1\wedge\cdots\wedge\xi_k$ and $q=\eta_1\wedge\cdots\wedge\eta_l$. Then we have that \begin{align*}
V_{[p,q]}&=\sum_{i,j}(-1)^{i+j}V_{[\xi_i,\eta_j]}\wedge V_{\xi_1}\wedge\cdots \widehat V_{\xi_i}\cdots\widehat V_{\eta_j}\wedge\cdots\wedge V_{\eta_l}\\
&=\sum_{i,j}-(-1)^{i+j}[V_{\xi_i},V_{\eta_j}]\wedge V_{\xi_1}\wedge\cdots \widehat V_{\xi_i}\cdots\widehat V_{\eta_j}\wedge\cdots \wedge V_{\eta_l}\\
&=-[V_p,V_q].
\end{align*}
Here we used the standard result of group actions that $[V_\xi,V_\eta]=-V_{[\xi,\eta]}$. The first claim now follows since $(\Rho_\g,[\cdot,\cdot])$ is a graded Lie algebra by Proposition \ref{Schouten is closed}. Moreover, we have that 

\begin{align*}
\partial V_p&=\partial (V_{\xi_1}\wedge\cdots\wedge V_{\xi_k})\\
&=\sum_{1\leq i<j\leq k}(-1)^{i+j}[V_{\xi_i},V_{\xi_j}]\wedge V_{\xi_1}\wedge\cdots\wedge\widehat V_{\xi_i}\wedge\cdots\wedge\widehat V_{\xi_j}\wedge\cdots\wedge V_{\xi_k}\\
&=-\sum_{1\leq i<j\leq k}(-1)^{i+j}V_{[\xi_i,\xi_j]}\wedge V_{\xi_1}\wedge\cdots\wedge\widehat V_{\xi_i}\wedge\cdots\wedge\widehat V_{\xi_j}\wedge\cdots\wedge V_{\xi_k}\\
&=-V_{\partial p}
\end{align*}

In particular then, if $p$ is in the Lie kernel, we have that $\partial V_p=-V_{\partial p}=0$.

\end{proof}

The last lemma of this section will be used repeatedly in the rest of the paper. We remark that it holds for arbitrary multivector fields; however, for our purposes it will suffice to consider the restriction to elements of $S$.

\begin{lemma}(\bf{Extended Cartan Lemma})\rm
\label{extended Cartan}
For decomposable $p=\xi_1\wedge\cdots\wedge\xi_k$ in $\Lambda^k\g$ and differential form $\tau$ we have that 

\begin{align*}
(-1)^kd(V_p\hk\tau)&=V_{\partial{p}}\hk\tau +\sum_{i=1}^k(-1)^i(V_{\xi_1}\wedge\cdots\wedge\widehat V_{\xi_i}\wedge\cdots\wedge V_{\xi_k})\hk\L_{V_{\xi_i}}\tau +V_p\hk d\tau.
\end{align*} 
\end{lemma}

\begin{proof}

This is Lemma 3.4 of \cite{ms} or Lemma 2.18 of \cite{cq}.
\end{proof}

While in this paper we will mostly be concerned with differential graded Lie algebras, we will also have the need to consider the more general structure of an $L_\infty$-algebra.

\subsection{$L_\infty$-algebras} We only state the definition of an $L_\infty$-algebra and do not go into detail. More detail can be found in \cite{rogers}, for example.

\begin{definition} An $L_\infty$-algebra is a graded vector space $L=\oplus_{i=-\infty}^\infty L_i$ together with a collection of  graded skew-symmetric linear maps  $\{l_k:L^{\otimes k}\to L \ ; \ k\geq 1\}$, with $\text{deg}(l_k)=k-2$, satisfying the following identity for all $m\geq 1$:

\[\sum_{\substack{ i+j=m+1\\ \sigma\in \text{Sh}(i,m-i)}}(-1)^\sigma\epsilon(\sigma)(-1)^{i(j-1)}l_j(l_i(x_{\sigma(1)},\ldots,x_{\sigma(i)}),x_{\sigma(i+1)},\ldots,x_{\sigma(m)})=0 .\] Here $\sigma$ is a permutation of $m$ letters, $(-1)^\sigma$ is the sign of $\sigma$, and $\epsilon(\sigma)$ is the Koszul sign. The subset Sh$(p,q)$ of permutations on $p+q$ letters is the set of $(p,q)$-unshuffles. A permutation $\sigma$ of $p+q$ letters is called a $(p,q)$-unshuffle if $\sigma(i)<\sigma(i+1)$ for $i\not=p$. 

An $L_\infty$-algebra $(L,\{l_k\})$ is called a Lie $n$-algebra if $L_i=0$ for $i\geq n$ and $i<0$.

\end{definition}

Since any differential graded Lie algebra is a Lie $n$-algebra (indeed, just take $l_1=\partial$, $l_2=[\cdot,\cdot]$ and $l_k=0$ for $k\geq 3$), Propositions \ref{Schouten is closed} and \ref{infinitesimal generator of Schouten} show that the spaces $S$ and $\Rho_\g$ have $L_\infty$-algebra structures.

\section{Multisymplectic Geometry}

We first recall some concepts from multisymplectic geometry.

\subsection{Multi-Hamiltonian Systems}

\begin{definition}
A manifold $M$ equipped with a closed $(n+1)$-form $\omega$ is called a pre-multisymplectic (or pre-$n$-plectic) manifold. If in addition the map $T_pM\to\Lambda^n T_p^\ast M,\  V\mapsto V\hk\omega$ is injective, then $(M,\omega)$ is called a multisymplectic  or $n$-plectic manifold.
\end{definition} 
\del{
The following proposition demonstrates the abundance of multisymplectic manifolds.

\begin{proposition}
A vector space of dimension $m$ admits an $n$-plectic form, with $n\geq 2$ if and only if $m\geq n+1$ and $m\not=n+2$. 
\end{proposition} 
\begin{proof}
This is proposition 2.6 of \cite{MS}.
\end{proof}
}

We will provide examples of multisymplectic manifolds in future sections, but for now give an example which comes up frequently in the rest of the paper.

\begin{example}\label{Multisymplectic Phase Space}\bf{(Multisymplectic Phase Space)}\rm

Let $N$ be a manifold and let $M=\Lambda^{k}(T^\ast N)$. Then $\pi:M\to N$ is a vector bundle over $N$ with canonical $k$-form $\theta\in\Omega^k(M)$ defined by \[\theta_{\mu_x}(Z_1,\cdots, Z_k) :=\mu_x(\pi_\ast(Z_1),\cdots,\pi_\ast(Z_k)), \] for $x\in N$, $\mu_x\in\Lambda^k(T^\ast_xN)$, and $Z_1,\cdots, Z_k \in T_{\mu_x}M$. The $(k+1)$-form $\omega\in\Omega^{k+1}(M)$ defined by $\omega=-d\theta$ is the canonical $(k+1)$-form. The pair $(M,\omega)$ is a $k$-plectic manifold.
\end{example}
\begin{definition}

If for $\alpha\in\Omega^{n-1}(M)$ there exists $X_\alpha\in\Gamma(TM)$ such that $d\alpha=-X_\alpha\hk \omega$ then we call $\alpha$ a Hamiltonian $(n-1)$-form and $X_\alpha$ its corresponding Hamiltonian vector field. We let $\Omega^{n-1}_{\text{Ham}}(M)$ denote the space of Hamiltonian $(n-1)$-forms. \end{definition}

\begin{remark}
If $\omega$ is $n$-plectic then the Hamiltonian vector field $X_\alpha$ is unique. If $\omega$ is pre-$n$-plectic then Hamiltonian vector fields are unique up to an element in the kernel of $\omega$. Also, notice that in the $1$-plectic (i.e. symplectic) case, every function is Hamiltonian.
\end{remark}

\begin{definition} In analogy to Hamiltonian mechanics, for a fixed $n$-plectic form $\omega$ and Hamiltonian $(n-1)$-form $H$, we call $(M,\omega,H)$ a multi-Hamiltonian system. We denote the Hamiltonian vector field of $H$ by $X_H$.
 \end{definition}

There are many examples of multi-Hamiltonian systems and we refer the reader to Section 3.1 of \cite{cq} for some results on their existence. 
\del{

\begin{example}
A orientable manifold $M$ with a volume form $\mu$. A corresponding Hamiltonian could be...We could take the acting group to be the volume preserving diffeomorphisms. 
\end{example}

\begin{example}
A G2 manifold where the corresponding Hamiltonian is...
\end{example}
We will study in more detail the following example

\begin{example}

The cotangent bundle with Hamiltonian

\end{example}

}

 In \cite{rogers} it was shown that to any multisymplectic manifold one can associate the following $L_\infty$-algebra.
\begin{definition}\label{Lie n observables}
The Lie $n$-algebra of observables, $L_\infty(M,\omega)$ is the following $L_\infty$-algebra. Let $L=\oplus_{i=0}^nL_i$ where $L_0=\Omega^{n-1}_{\text{Ham}}(M)$ and $L_i=\Omega^{n-1-i}(M)$ for $1\leq i\leq n-1$. The maps $l_k:L^{\otimes k}\to L$ of degree $k-2$ are defined as follows. The map $l_1$ is defined to be the exterior derivative on elements of positive degree, and $0$ on Hamiltonian $(n-1)$-forms. For $k>1$ the maps $l_k$ are defined to be $0$ on elements with positive degree and to be $l_k(\alpha_1,\cdots,\alpha_k):=\zeta(k)X_{\alpha_k}\hk\cdots \hk X_{\alpha_1}\hk\omega$ if all of $\alpha_1,\ldots,\alpha_k$ are Hamiltonian $(n-1)$-forms.
 Here $\zeta(k)$ is defined to equal $-(-1)^{\frac{k(k+1)}{2}}$. We introduce this notation as this sign comes up frequently.
\end{definition}

\begin{remark}\label{ugly signs}
It is easily verified that $\zeta(k)\zeta(k+1)=(-1)^{k+1}$. For future reference we also note that $\zeta(k)\zeta(l)\zeta(k+l-1)=-(-1)^{k+l+kl}$ and $\zeta(k)\zeta(l)=-(-1)^{lk}\zeta(k+l).$
\end{remark}

The following lemma from \cite{rogers} will be useful later on in the paper.

\begin{lemma}
\label{rogers 3.7}
Let $\alpha_1,\ldots,\alpha_m\in\Omega^{n-1}_{\text{Ham}}(M)$ be arbitrary Hamiltonian $(n-1)$-forms on a multisymplectic manifold $(M,\omega)$. Let $X_1,\ldots, X_m$ denote the associated Hamiltonian vector fields. Then
\[d(X_m\hk\cdots\hk X_1\hk\omega)=(-1)^m\sum_{1\leq i<j\leq m}(-1)^{i+j}X_m\hk\cdots\hk\widehat X_j\hk\cdots\hk\widehat X_i\hk\cdots\hk X_1\hk[X_i,X_j]\hk\omega.\]
\end{lemma}

\begin{proof}
This is Lemma 3.7 of \cite{rogers}.
\end{proof}
Lastly we recall the terminology for group actions on a multisymplectic manifold. 

\begin{definition}
A Lie group action $\Phi:G\times M\to M$ is called multisymplectic if $\Phi_g^\ast\omega=\omega$. A Lie algebra action $\g\times \Gamma(TM)\to \Gamma(TM)$ is called multisymplectic if $\L_{V_\xi}\omega=0$ for all $\xi\in\g$. We remark that a multisymplectic Lie group action induces a multisymplectic Lie algebra action. Conversely, a multisymplectic Lie algebra action induces a multisymplectic group action if the Lie group is connected.  Moreover, as in \cite{cq}, we will call a Lie group action on a multi-Hamiltonian system $(M,\omega,H)$ locally, globally, or strictly $H$-preserving if it is multisymplectic and if $\L_{V_\xi} H$ is closed, exact, or zero respectively, for all $\xi\in\g$. 
\end{definition}

\subsection{Homotopy Co-Momentum Maps}
For a group acting on a symplectic manifold $M$, a moment map is a Lie algebra morphism between $(\g,[\cdot,\cdot])$ and  $(C^\infty(M),\{\cdot,\cdot\})$, where $\{\cdot,\cdot\}$ is the Poisson bracket. In multisymplectic geometry, the $n$-plectic form no longer provides a Lie algebra structure on the space of smooth functions. However, as we saw in the previous section, the $n$-plectic structure does define an $L_\infty$-algebra, namely the Lie $n$-algebra of observables. A homotopy co-momentum map is an $L_\infty$-morphism from $\g$ to the Lie $n$-algebra of observables. We explain what this means in the following definition, while refer the reader to \cite{questions} for further information on $L_\infty$-morphisms.

For the rest of this section, we assume a multisymplectic action of a Lie algebra $\g$ on $(M,\omega)$.
\begin{definition}\label{hmm}
A (homotopy) co-momentum map is an $L_\infty$-morphism $(f)$ between $\g$ and the Lie $n$-algebra of observables. This means that $(f)$ is a collection of maps $f_1:\Lambda^1\g\to\Omega^{n-1}_{\text{Ham}}(M)$ and $f_k:\Lambda^k\g\to\Omega^{n-k}(M)$ for $k\geq 2$ satisfying, \[df_1(\xi)=-V_\xi\hk\omega\] for $\xi\in\g$ and  \begin{equation}\label{hcmm}-f_{k-1}(\partial p)=df_k(p)+\zeta(k)V_p\hk\omega,\end{equation} for $p\in\Lambda^k\g$ and $k\geq 1$. A co-momentum map is called equivariant if each component $f_i:\Lambda^i\g\to\Omega^{n-i}(M)$ is equivariant with respect to the adjoint and pullback actions respectively.
\end{definition}

For reasons to be discussed in detail in the next section, we will mostly be concerned with the restriction of co-momentum maps to the Lie kernel. Under this assumption, equation $(\ref{hcmm})$ reduces to \begin{equation}\label{hcmm kernel}df_k(p)=-\zeta(k)V_p\hk\omega.\end{equation}

\begin{definition} A weak (homotopy) co-momentum map $(f)$ is a collection of maps $f_k:\Rho_{\g,k}\to\Omega^{n-k}_{\mathrm{Ham}}(M)$, where $1\leq k\leq n$, satisfying equation (\ref{hcmm kernel}).

\end{definition}

\begin{remark}
By restricting the domain of a homotopy co-momentum map to the Lie kernel, we see that the multi-moment maps of Madsen and Swann (in \cite{ms} and \cite{MS}) are given precisely by the $n$-th component of our weak homotopy co-momentum maps.
\end{remark}

We conclude this section with some examples of weak co-momentum maps: 

In Hamiltonian mechanics, the phase space of a manifold $M$ is the symplectic manifold $(T^\ast M,\omega=-d\theta)$. The next example generalizes this to the setting of multisymplectic geometry.
 
 \begin{example}\bf{(Multisymplectic Phase Space)}\rm
 \label{multisymplectic phase space}
 As in Example \ref{Multisymplectic Phase Space}, let $N$ be a manifold and let $M=\Lambda^{k}(T^\ast N)$, with $\pi:M\to N$ the projection map. Let $\theta$ and $\omega=-d\theta$ denote the canonical $k$ and $(k+1)$-forms respectively.  Let $G$ be a group acting on $N$ and lift this action to $M$ in the standard way. Such an action on $M$ necessarily preserves $\theta$. We define a weak homotopy co-momentum map by \[f_l(p):=-\zeta(l+1)V_p\hk\theta,\]for $p\in\Rho_{\g,l}$.

We now show that $(f)$ is a weak homotopy co-momentum map. For $l\geq 1$, first consider a decomposable element $p=A_1\wedge\cdots\wedge A_l$ in $\Lambda^l\g$. Then,

\begin{align*}
df_l(p)&=-\zeta(l+1)d(V_p\hk\theta)\\
&=-\zeta(l+1)(-1)^{l}\left(\partial V_p\hk\theta +\sum_i^l(-1)^iA_1\wedge\cdots\wedge\widehat A_i\wedge\cdots\wedge A_l\hk \L_{A_i}\theta-V_p\hk\omega\right)&\text{by Lemma \ref{extended Cartan}}\\
&=\zeta(l)\left(\partial V_p\hk\theta -V_p\hk\omega\right) &\text{since $G$ preserves $\theta$}
\end{align*}

By linearity, we thus see that this equation holds for an arbitrary element in $\Lambda^l\g$. That is, for all $p\in\Lambda^l\g$, \[df_l(p)=\zeta(l)\left(\partial V_p\hk\theta -V_p\hk\omega\right).\] If we assume now that $p\in\Rho_{\g,l}$, it then follows from Proposition \ref{infinitesimal generator of Schouten} that  \[df_l(p)=-\zeta(l)V_p\hk\omega.\] Thus by equation $(\ref{hcmm kernel})$ we see $(f)$ is a weak homotopy co-momentum map.
\end{example}

\begin{remark}
In symplectic geometry, symmetries on the phase space $T^\ast M$ have an important relationship with the classical momentum and position functions (see Chapter 4.3 of \cite{Marsden}). These momentum and position functions satisfy specific commutation relations which play an important role in connecting classical and quantum mechanics. Once we extend the Poisson bracket to multisymplectic manifolds and discuss a generalized notion of symmetry, we will come back to this multisymplectic phase space and give a generalization of these classical momentum and position functions (see Section 6.1).
\end{remark}
 
The next two examples will be used when we look at manifolds with a torsion-free $G_2$ structure.

\begin{example}\bf{($\C^3$ with the standard holomorphic volume form)}\rm
\label{complex moment map 1}
Consider $\C^3$ with standard coordinates $z_1,z_2,z_3$. Let $\Omega=dz^1\wedge dz^2\wedge dz^3$ denote the standard holomorphic volume form. Let \[\alpha=\text{Re}(\Omega)=\frac{1}{2}(dz^1\wedge dz^2\wedge dz^3+d\o z^1\wedge d\o z^2\wedge d\o z^3).\] It follows that $\alpha$ is a $2$-plectic form on $\C^3$.  We consider the diagonal action by the maximal torus $T^2\subset SU(3)$ given by $(e^{i\theta},e^{i\eta})\cdot(z_1,z_2,z_3)=(e^{i\theta}z_1,e^{i\eta}z_2,e^{-i(\theta+\eta)}z_3)$. We have $\mathfrak{t}^2=\R^2$ and that the infinitesimal generators of $(1,0)$ and $(0,1)$ are \[A=\frac{i}{2}\left(z_1\frac{\pd}{\pd z_1}-z_3\frac{\pd}{\pd z_3}-\o z_1\frac{\pd}{\pd\o z_1}+\o z_3\frac{\pd}{\pd \o z_3}\right)\]and \[B=\frac{i}{2}\left(z_2\frac{\pd}{\pd z_2}-z_3\frac{\pd}{\pd z_3}-\o z_2\frac{\pd}{\pd\o z_2}+\o z_3\frac{\pd}{\pd \o z_3}\right)\] respectively.

A computation then shows that 
\[A\hk\alpha=\frac{1}{2}d(\text{Im}(z_1z_3dz^2))\]
and
\[B\hk\alpha=\frac{1}{2}d(\text{Im}(z_1z_3dz^1)).\]

Moreover,
\[B\hk A\hk\alpha=-\frac{1}{4}d(\text{Re}(z_1z_2z_3))\]
Since $G=T^2$ is abelian we have that $\Rho_{\g,2}=\Lambda^2\g$. Thus, by equation $(\ref{hcmm kernel})$ we see a weak homotopy co-momentum map is given by

\[f_1(A)=\frac{1}{2}(\text{Im}(z_1z_3dz^2)) \ \ \ \ \ \ \ \ \ \ f_1(B)=\frac{1}{2}(\text{Im}(z_1z_3dz^1))\]
and
\[f_2(A\wedge B)=\frac{1}{4}(\text{Re}(z_1z_2z_3)).\]

If instead of Re$(\Omega)$ we were to consider Im$(\Omega)$ then in the above expressions for $f_1$ and $f_2$ we would just swap the roles of Re and Im.
\end{example}

\begin{example}\bf{($\C^3$ with the standard Kahler form)}\rm
\label{complex moment map 2}
Working with the same set up as Example \ref{complex moment map 1}, now consider the standard Kahler form $\omega=\frac{i}{2}(dz^1\wedge d\o z^1+dz^2\wedge d\o z^2+dz^3\wedge d\o z^3)$. This is a $1$-plectic (i.e. symplectic) form on $\C^3$.  A computation shows that \[A\hk\omega=-\frac{1}{4}d(|z_1|^2-|z_3|^2)\] and \[B\hk\omega=-\frac{1}{4}d(|z_2|^2-|z_3|^2).\] Thus, by equation $(\ref{hcmm kernel})$ a weak homotopy co-momentum map is given by

\[f_1(A)=-\frac{1}{4}(|z_1|^2-|z_3|^2) \ \ \ \ \ \ \ \ \ \ f_1(B)=-\frac{1}{4}(|z_2|^2-|z_3|^2).\]

\end{example}

For our last example of this section, we consider a multi-Hamiltonian system which models the motion of a particle, with unit mass, under no external net force. %That is, in the following system, the particle is following a geodesic. 
\begin{example}\label{translation}\bf{(Motion in a conservative system under translation)}\rm

Consider $\R^3$ with the standard metric $g$ and standard coordinates $q^1$, $q^2$, $q^3$. Let $q^1$, $q^2$, $q^3$, $p_1$, $p_2$, $p_3$ denote the induced coordinates on $T^\ast\R^3=\R^6$.
The motion of a particle in $\R^3$, subject to no external force, is given by a geodesic. That is, the path $\gamma$ of the particle is an integral curve for the geodesic spray $S$, a vector field on $T\R^3=\R^6$.  Using the metric to identity $T\R^3$ and $T^\ast\R^3$, the geodesic spray is given by \[S=g^{kj}p_j\frac{\pd}{\pd q^k}-\frac{1}{2}\frac{\pd g^{ij}}{\pd q^k}p_ip_j\frac{\pd}{\pd p_k},\] as shown in Example 5.21 of \cite{me}. Since we our working with the standard metric, the geodesic spray is just \[S=\sum_{i=1}^3p_i\frac{\pd}{\pd q^i}.\] 

Let $M=T^\ast\R^3=\R^6$ and consider the multi-Hamiltonian system $(M,\omega,H)$ where \[\omega=\text{vol}=dq^1dq^2dq^3dp_1dp_2dp_3\] is the canonical volume form, and \[H=\frac{1}{2}\left((p_1q^2dq^3-p_1q^3dq^2)-(p_2q^1dq^3-p_2q^3dq^2)+(p_3q^1dq^2-p_3q^2dq^1)\right)dp_1dp_2dp_3.\] Then \[S\hk\omega=dH\] so that the $X_H=S$. That is, the Hamiltonian vector field in this multi-Hamiltonian system is the geodesic spray.  Consider the translation action of $G=\R^3$ on $\R^3$ and pull this back to an action on $M$. The infinitesimal generators of $e_1,e_2,e_3$ on $M$ are $\frac{\pd}{\pd q^1},\frac{\pd}{\pd q^2}$ and $\frac{\pd}{\pd q^3}$ respectively.  We compute the moment map for this action:

Since $\frac{\pd}{\pd q^1}\hk\omega=dq^2dq^3dp_1dp_2dp_3$ it follows that  $f_1(e_1)=\frac{1}{2}(q^2dq^3-q^3dq^2)dp_1dp_2dp_3$ satisfies $df(e_1)=V_{e_1}\hk\omega$. Similar computations show that the following is a homotopy co-momentum map for the translation action on $(M,\omega,H)$:

\[f_1(e_1)=\frac{1}{2}(q^2dq^3-q^3dq^2)dp_1dp_2dp_3, \] \[ f_1(e_2)=\frac{1}{2}(q^1dq^3-q^3dq^1)dp_1dp_2dp_3, \] \[f_1(e_3)=\frac{1}{2}(q^1dq^2-q^2dq^2)dp_1dp_2dp_3,\]
\[f_2(e_1\wedge e_2)=q^3dp_1dp_2dp_3, \ \  \ \ \ \ f_2(e_1\wedge e_3)=q^2dp_1dp_2dp_3, \ \ \ \ \ \ f_2(e_2\wedge e_3)=q^1dp_1dp_2dp_3,\]and
\[f_3(e_1\wedge e_2\wedge e_3)=\frac{1}{3}\left(p_1dp_2dp_3+p_2dp_3dp_1+p_3dp_1dp_2\right).\]

\end{example}
\begin{remark}
In Section 5.3 we will  come back to Example \ref{translation} and consider the multisymplectic symmetries and conserved quantities coming from this homotopy co-momentum map.
\end{remark}
\section{Multisymplectic Symmetries and Conserved Quantities}

In this section we give a definition of conserved quantities and continuous symmetries on multisymplectic manifolds. In symplectic geometry, the Poisson bracket plays a large role in the discussion of conserved quantities. To that end, we first try to generalize the Poisson bracket to multisymplectic geometry. \del{We will always work with a fixed multi-Hamiltonian system $(M,\omega,H)$. We let $X_H$ denote the unique Hamiltonian vector field corresponding to $H$. }

\subsection{A Generalized Poisson Bracket}
We first extend the notion of a Hamiltonian $(n-1)$-form to arbitrary forms of degree $\leq n-1$.
\begin{definition}We call
 \[\Omega^{n-k}_{\text{Ham}}(M):=\{\alpha\in\Omega^{n-k}(M) ; \text{ there exists $X_\alpha\in\Gamma(\Lambda^k(TM))$ with $d\alpha=-X_\alpha\hk\omega$ } \} \]the set of Hamiltonian $(n-k)$-forms. For a Hamiltonian $(n-k)$-form $\alpha$, we call $X_\alpha$ a corresponding Hamiltonian $k$-vector field (or multivector field if $k$ is not explicit).
 
We call 
 \[\X^k_{\text{Ham}}(M) := \{X\in\Gamma(\Lambda^k(TM)) ; X\hk\omega \text{ is exact} \} \] the set of Hamiltonian $k$-vector fields. We will refer to a primitive of $X\hk\omega$ as a corresponding Hamiltonian $(n-k)$-form. 
 
 \end{definition}

Of course, given a Hamiltonian $(n-k)$-form, it does not necessarily have a unique associated Hamiltonian multivector field. Moreover, a Hamiltonian $k$-vector field doesn't necessarily have a unique corresponding Hamiltonian $(n-k)$-form. However, the following is clear:
\begin{proposition}
\label{kernel}
For $\alpha\in\Omega^{n-k}_{\mathrm{Ham}}(M)$,  any two of its Hamiltonian $k$-vector fields differ by something in the kernel of $\omega$. Conversely, for $X\in\X^k_{\mathrm{Ham}}(M)$, any two of its Hamiltonian forms differ by a closed form. 
\end{proposition}

Proposition \ref{kernel} motivates consideration of the following spaces. Let  $\widetilde{\X}^k_{\text{Ham}}(M)$ denote the quotient space of $\X^k_{\text{Ham}}(M)$ by elements in the kernel of $\omega$. Let $\widetilde{\Omega}^{n-k}_{\text{Ham}}(M)$ denote the quotient of $\Omega^{n-k}_{\text{Ham}}$ by closed forms. We let \[\Omega_{\text{Ham}}(M)=\oplus_{k=0}^{n-1}\Omega^k_{\text{Ham}}(M)\]and \[\widetilde\Omega_{\text{Ham}}(M)=\oplus_{k=0}^{n-1}\widetilde\Omega^k_{\text{Ham}}(M).\] Similarly, we let \[\X_{\text{Ham}}(M)=\oplus_{k=0}^{n-1}\X^k_{\text{Ham}}(M)\] and \[\widetilde\X_{\text{Ham}}(M)=\oplus_{k=0}^{n-1}\widetilde\X^k_{\text{Ham}}(M).\] It is clear that the map from $\widetilde{\Omega}^{n-k}_{\text{Ham}}(M)$ to $\widetilde{\X}^k_{\text{Ham}}(M)$ given by $[\alpha]\mapsto [X_\alpha]$ is a bijection.

\begin{proposition}
 The spaces $\widetilde{\Omega}^{n-k}_{\mathrm{Ham}}(M)$ and $\widetilde{\X}^k_{\mathrm{Ham}}(M)$ are isomorphic. 
\end{proposition}

Later on, we will see that there are graded Lie brackets on the spaces $\widetilde{\Omega}_{\text{Ham}}(M)$ and $\widetilde{\X}_{\text{Ham}}(M)$ making them isomorphic as graded Lie algebras.

The next proposition will be used to show that certain statements about a Hamiltonian form are independent of the choice of the corresponding Hamiltonian multivector field.
\begin{proposition}
\label{well defined}
If $\kappa$ is in the kernel of $\omega$, then for any $X\in\X^k_{\mathrm{Ham}}(M)$ we have $[X,\kappa]\hk\omega=0$.
\end{proposition}
\begin{proof}
Using equation (\ref{bracket hook}) together with the fact that $\L_X\omega=0$ and $\kappa\hk\omega=0$ we have \[[X,\kappa]\hk\omega=(-1)^{k(k+1)}\L_{X}(\kappa\hk\omega)-\kappa\hk\L_{X}\omega=0.\]
\end{proof}

The next proposition shows that, as in symplectic geometry, any Hamiltonian multivector field preserves $\omega$.

\begin{proposition}
\label{preserve}
For $\alpha\in\Omega^{n-k}_{\mathrm{Ham}}(M)$ we have that $\L_{X_\alpha}\omega=0$ for all Hamiltonian multivector fields $X_\alpha$ of $\alpha$.
\end{proposition}

\begin{proof}
Let $X_\alpha$ be a Hamiltonian multivector field. We have that \begin{align*}
\L_{X_\alpha}\omega&=\L_{X_\alpha}\omega\\
&=d(X_\alpha\hk\omega)-(-1)^kX_\alpha\hk d\omega&\text{by equation (\ref{Lie})}\\
&=d(X_\alpha\hk\omega)&\text{since $d\omega=0$}\\
&=-d(d\alpha)&\text{by definition}\\
&=0.
\end{align*}
\end{proof}

\del{
The next lemma is a generalization of Lemma 3.7 of \cite{rogers}. It will be used repeatedly throughout this paper.

\begin{lemma}
\label{Rogers}
Fix $m\geq 2$. Let $\alpha_1,\ldots,\alpha_m$ be Hamiltonian forms of degree $n-k_1,\ldots, n-k_m$ respectively. Let $X_1,\ldots, X_m$ be arbtirary Hamiltonian vector fields for $\alpha_1,\ldots,\alpha_m$. Note that $X_j$ has degree $k_j$. The following holds:
\begin{align*}&d(X_m\hk\cdots\hk X_1\hk\omega)=\\
&(-1)^m(-1)^{k_1\cdots k_m}\sum_{1\leq i<j\leq m}(-1)^{i+j}X_m\hk\cdots\hk\widehat X_i\hk\cdots\hk\widehat X_j\hk\cdots\hk X_1\hk[X_i,X_j]\hk\omega
\end{align*}
\end{lemma}
\begin{proof}
Brutal
\end{proof}
}

We now put in a structure analogous to the Poisson bracket in Hamiltonian mechanics, which has graded analogous properties. 

Given $\alpha\in\Omega^{n-k}_{\text{Ham}}(M)$ and $\beta\in\Omega^{n-l}_{\text{Ham}}(M)$, a first attempt would be to define their generalized bracket to be \[\{\alpha,\beta\} := X_\beta\hk X_\alpha\hk \omega,\] mimicing the Poisson bracket in symplectic geometry. However, we can see right away that this bracket is not graded anti-commutative since $\{\alpha,\beta\}=(-1)^{kl}\{\beta,\alpha\}$. Hence, we modify our grading of the Hamiltonian forms, following the work done in \cite{dropbox}.

\begin{definition}
Let $\mathcal{H}^p(M)=\Omega^{n-p+1}_{\text{Ham}}(M)$. That is, we are assigning the grading of $\alpha\in\Omega^{n-k}_{\text{Ham}}(M)$ to be $|\alpha|=k+1$. For $\alpha\in\Omega^{n-k}_{\text{Ham}}(M)$ and $\beta\in\Omega^{n-l}_{\text{Ham}}(M)$ (i.e. $\alpha\in\mathcal{H}^{k+1}(M)$ and $\beta\in\mathcal{H}^{l+1}(M)$) we define their (generalized) Poisson bracket to be \begin{align*}\{\alpha,\beta\}&:=(-1)^{|\beta|}X_\beta\hk X_\alpha\hk\omega\\
&=(-1)^{l+1}X_\beta\hk X_\alpha\hk\omega.
\end{align*}

Notice that this bracket is well defined follows directly from Proposition \ref{kernel}.
\end{definition}

With this new grading, the generalized Poisson bracket is graded commutative. 

\begin{proposition}
\label{skew graded}
 Let $\alpha$ be a form of grading $|\alpha|=k+1$ and $\beta$ a form of grading $|\beta|=l+1$. That is,  $\alpha\in\Omega^{n-k}_{\text{Ham}}(M)$ and $\beta\in\Omega^{n-l}_{\text{Ham}}(M)$. Then we have that \[\{\alpha,\beta\}=-(-1)^{|\alpha||\beta|}\{\beta,\alpha\}.\]

\end{proposition}

\begin{proof}
By definition,
\begin{align*}
\{\alpha,\beta\}&=(-1)^{l+1}X_\beta\hk X_\alpha\hk\omega\\
&=(-1)^{l+1}(-1)^{kl}X_\alpha\hk X_\beta\hk\omega\\
&=-(-1)^{l(k+1)}X_\alpha\hk X_\beta\hk\omega\\
&=-(-1)^{(l+1)(k+1)+k+1}X_\alpha\hk X_\beta\hk\omega\\
&=-(-1)^{|\alpha||\beta|}(-1)^{k+1}X_\alpha\hk X_\beta\hk\omega\\
&=-(-1)^{|\alpha||\beta|}\{\beta,\alpha\}.
\end{align*}
\end{proof}

The next lemma shows that the bracket of two Hamiltonian forms is Hamiltonian. In symplectic geometry, we have $X_{\{f,g\}}=[X_f,X_g]$ (or $X_{\{f,g\}}=-[X_f,X_g]$ if the defining equation for a Hamiltonian vector field is $X_\alpha\hk\omega=d\alpha$). In multisymplectic geometry we have

\begin{lemma}
\label{Poisson is Schouten}
For $\alpha\in\Omega^{n-k}_{\text{Ham}}(M)$ and $\beta\in\Omega^{n-l}_{\text{Ham}}(M)$ their bracket $\{\alpha,\beta\}$ is in $\Omega^{n+1-k-l}_{\text{Ham}}(M)$. That is, $\{\alpha,\beta\}$ is a Hamiltonian form with grading $|\{\alpha,\beta\}|=k+l-2$. More precisely, we have that  $[X_\alpha,X_\beta]$ is a Hamiltonian vector field for $\{\alpha,\beta\}$.
\end{lemma}

\begin{proof}
We have that
\begin{align*}
[X_\alpha,X_\beta]\hk\omega&=-X_\beta\hk d(X_\alpha\hk\omega)+(-1)^ld(X_\beta\hk X_\alpha\hk\omega)\\
&\quad{}+(-1)^{kl+k}X_\alpha\hk X_\beta\hk dw-(-1)^{kl+k+l}X_\alpha\hk d(X_\beta\hk\omega)&\text{by equation (\ref{interior equation})}\\
&=(-1)^ld(X_\beta\hk X_\alpha\hk\omega)\\
&=-d(\{\alpha,\beta\}).
\end{align*}

\end{proof}

We now investigate the Jacobi identity for this bracket. In \cite{dropbox} it was mentioned that the graded Jacobi identity holds up to a closed form. We now show that the graded Jacobi identity holds up to an exact term. 
\begin{proposition}\bf{(Graded Jacobi.)} \rm 
\label{Jacobi} \ Fix $\alpha\in\Omega^{n-k}_{\text{Ham}}(M)$, $\beta\in\Omega^{n-l}_{\text{Ham}}(M)$ and $\gamma\in\Omega^{n-m}_{\text{Ham}}(M)$. Let $X_\alpha, X_\beta$ and $X_\gamma$ denote arbitrary Hamiltonian multivector fields for $\alpha,\beta$ and $\gamma$ respectively. Then we have that
\[\sum_{\text{cyclic}}(-1)^{|\alpha||\gamma|}\{\alpha,\{\beta,\gamma\}\} = (-1)^{|\beta||\gamma|+|\beta||\alpha|+|\beta|}d(X_\alpha\hk X_\beta\hk X_\gamma\hk\omega).\] 

\end{proposition}
\begin{proof}
By definition, we have that \[\{\alpha,\beta\}=(-1)^{|\beta|}X_\beta\hk X_\alpha\hk\omega=(-1)^{|\beta|+1}X_\beta\hk d\alpha.\] Since $X_\beta$ is in $\Lambda^{|\beta|+1}(TM)$, by (\ref{Lie}) it follows that \begin{align}\{\alpha,\beta\}&\nonumber=(-1)^{|\beta|+1}(-1)^{|\beta|+1}(d(X_\beta\hk\alpha)-\L_{X_\beta}\alpha)\\&\label{bracket 1}=d(X_\beta\hk\alpha)-\L_{X_\beta}\alpha.\end{align}Thus, \begin{align*}
\{\alpha,\{\beta,\gamma\}\}&=(-1)^{|\beta||\gamma|+1}\{\alpha,\{\gamma,\beta\}\}\\
&=(-1)^{|\beta||\gamma|+1+(|\beta|+|\gamma|)|\alpha|+1}\{\{\gamma,\beta\},\alpha\}\\
&=(-1)^{|\beta||\gamma|+1+(|\beta|+|\gamma|)|\alpha|+1}(d(X_\alpha\hk\{\gamma,\beta\})-\L_{X_\alpha}\{\gamma,\beta\})&\text{by $(\ref{bracket 1})$}\\
&=(-1)^{|\beta||\gamma|+|\beta||\alpha|+|\gamma||\alpha|}(d(X_\alpha\hk\{\gamma,\beta\})-\L_{X_\alpha}(d(X_\beta\hk\gamma))+\L_{X_\alpha}\L_{X_\beta}\gamma).
\end{align*} Hence, \begin{equation}\label{first Jacobi}(-1)^{|\alpha||\gamma|}\{\alpha,\{\beta,\gamma\}\}=(-1)^{|\beta||\gamma|+|\beta||\alpha|}\left(d(X_\alpha\hk\{\gamma,\beta\})-\L_{X_\alpha}d(X_\beta\hk\gamma)+\L_{X_\alpha}\L_{X_\beta}\gamma\right).\end{equation} Similarly, since $|\{\gamma,\alpha\}|=|\gamma|+|\alpha|-2$, we have that \begin{align*}\{\beta,\{\gamma,\alpha\}\}&=(-1)^{(|\gamma|+|\alpha|)|\beta|+1}\{\{\gamma,\alpha\},\beta\}\\&=(-1)^{|\gamma||\beta|+|\alpha||\beta|+1}(d(X_\beta\hk\{\gamma,\alpha\})-\L_{X_\beta}d(X_\alpha\hk\gamma)+\L_{X_\beta}\L_{X_\alpha}\gamma).&\text{by $(\ref{bracket 1})$}\\
\end{align*}Hence, \begin{equation}\label{second Jacobi}(-1)^{|\beta||\alpha|}\{\beta,\{\gamma,\alpha\}\}=(-1)^{|\gamma||\beta|+1}\left(d(X_\beta\hk\{\gamma,\alpha\})-\L_{X_\beta}d(X_\alpha\hk\gamma)+\L_{X_\beta}\L_{X_\alpha}\gamma\right).\end{equation} Lastly, using Lemma \ref{Poisson is Schouten} and $(\ref{bracket 1})$, we have that  \begin{equation}\label{third Jacobi} (-1)^{|\gamma||\beta|}\{\gamma,\{\alpha,\beta\}\}=(-1)^{|\gamma||\beta|}\left(d([X_\alpha,X_\beta]\hk\gamma)-\L_{[X_\alpha,X_\beta]}\gamma\right).\end{equation} Now we notice that by (\ref{L bracket}) the terms involving $\L_{X_\alpha}\L_{X_\beta}\gamma$ from (\ref{first Jacobi}) , $\L_{X_\beta}\L_{X_\alpha}\gamma$ from (\ref{second Jacobi}) and $\L_{[X_\alpha,X_\beta]}\gamma$ from (\ref{third Jacobi}) add to zero. Hence we now consider the term $(-1)^{|\gamma||\beta|}d([X_\alpha,X_\beta]\hk\gamma)$ from (\ref{third Jacobi}). We have that

\begin{align*}
d(&[X_\alpha,X_\beta]\hk\gamma)=d\left((-1)^{|\alpha|(|\beta|+1)}\L_{X_\alpha}(X_\beta\hk\gamma)-X_\beta\hk\L_{X_\alpha}\gamma\right)&\text{by (\ref{bracket hook})}\\
&=(-1)^{|\alpha|(|\beta|+1)}d\left(\L_{X_\alpha}(X_\beta\hk\gamma)\right)-d\left(X_\beta\hk d(X_\alpha\hk\gamma)\right)+(-1)^{|\alpha|+1}d\left(X_\beta\hk X_\alpha\hk d\gamma\right)&\text{by (\ref{Lie})}\\
&=(-1)^{|\alpha||\beta|}\L_{X_\alpha}d(X_\beta\hk\gamma)-\L_{X_\beta} d(X_\alpha\hk\gamma)+(-1)^{|\alpha|}d(X_\beta\hk X_\alpha\hk X_\gamma\hk\omega).&\text{by (\ref{dL}) and (\ref{Lie})}
\end{align*}
Thus \begin{equation}\label{fourth Jacobi}\begin{aligned}
(-1)^{|\gamma||\beta|}d([X_\alpha,X_\beta]\hk\gamma)&=(-1)^{|\alpha||\beta|+|\gamma||\beta|}\L_{X_\alpha}d(X_\beta\hk\gamma)-(-1)^{|\gamma||\beta|}\L_{X_\beta}d(X_\alpha\hk\gamma)+\\
&\quad{}+(-1)^{|\gamma||\beta|+|\alpha|}d(X_\beta\hk X_\alpha\hk X_\gamma\hk\omega).
\end{aligned}\end{equation}Thus, upon adding (\ref{first Jacobi}), (\ref{second Jacobi}) and (\ref{fourth Jacobi}) we are left with 

\begin{align*}
(-1)^{|\alpha||\gamma|}&\{\{\alpha,\beta\},\gamma\}+(-1)^{|\beta||\alpha|}\{\{\beta,\gamma\},\alpha\}+(-1)^{|\gamma||\beta|}\{\{\gamma,\alpha\},\beta\}  \\
&=(-1)^{|\beta||\gamma|+|\beta||\alpha|}d(X_\alpha\hk\{\gamma,\beta\})+(-1)^{|\gamma||\beta|+1}d(X_\beta\hk\{\gamma,\alpha\})\\
&\quad{}+(-1)^{|\gamma||\beta|+|\alpha|}(X_\beta\hk X_\alpha\hk X_\gamma\hk\omega)\\
&=(-1)^{|\beta||\gamma|+|\beta||\alpha|+|\beta|}d(X_\alpha\hk X_\beta\hk X_\gamma\hk\omega)+(-1)^{|\gamma||\beta|+1+|\alpha|}d(X_\beta\hk X_\alpha\hk X_\gamma\hk\omega)\\
&\quad{} +(-1)^{|\gamma||\beta|+|\alpha|}d(X_\beta\hk X_\alpha\hk X_\gamma\hk\omega)\\
&=(-1)^{|\beta||\gamma|+|\beta||\alpha|+|\beta|}d(X_\alpha\hk X_\beta\hk X_\gamma\hk\omega).
\end{align*}

\end{proof}

Summing up the results of this section we have confirmed Theorem 4.1 of \cite{dropbox}: 

\begin{proposition}
\label{graded Lie algebra of forms}
With the above grading, $(\widetilde\Omega_{\text{Ham}}(M),\{\cdot,\cdot\})$ is a graded Lie algebra.
\end{proposition}

\begin{proof}
The bracket is well defined on $\widetilde\Omega_{\text{Ham}}(M)$ since if $\gamma$ is closed then $\{\gamma,\alpha\}=(-1)^kX_\alpha\hk d\gamma=0$. Clearly the bracket is bilinear. Proposition \ref{skew graded} shows that the bracket is skew graded and Proposition \ref{Jacobi} shows that it satisfies the Jacobi identity.
\end{proof}

\subsection{Conserved Quantities and their Algebraic Structure}
We now turn our attention towards conserved quantities. In symplectic geometry, a conserved quantity is a $0$-form $\alpha$ that is preserved by the Hamiltonian, i.e. satisfying $\L_{X_H}\alpha=0$. A generalization of this definition to multisymplectic geometry was given in \cite{cq}; however, we add the requirement that a conserved quantity is also Hamiltonian. By adding in this requirement, we can now take the generalized Poisson bracket of two conserved quantities, as in symplectic geometry. 

We work with a fixed multi-Hamiltonian system $(M,\omega,H)$ with $\omega\in\Omega^{n+1}(M)$ and $H\in\Omega^{n-1}_{\text{Ham}}(M)$, and let $X_H$ denote the corresponding Hamiltonian vector field.

\begin{definition}
A Hamiltonian $(n-k)$-form $\alpha$ in $\Omega^{n-k}_{\text{Ham}}(M)$ is called 
\begin{itemize}
\item locally conserved if $\L_{X_H}\alpha$ is closed,
\item globally conserved if $\L_{X_H}\alpha$ is exact, 
\item strictly conserved if $\L_{X_H}\alpha=0$.
\end{itemize}

As in \cite{cq}, we denote the space of locally, globally and strictly conserved forms by $\mathcal{C}_{\text{loc}}(X_H)$, $\mathcal{C}(X_H)$ and $\mathcal{C}_{\text{str}}(X_H)$ respectively.  We will let $\widetilde{\mathcal{C}}_{\text{loc}}(X_H)$, $\widetilde{\mathcal{C}}(X_H)$ and $\widetilde{\mathcal{C}}_{\text{str}}(X_H)$ denote the conserved quantities modulo closed forms. %Note that, by definition, $\widetilde{\mathcal{C}}(X_H)=\widetilde{\mathcal{C}}_{loc}(X_H)=\widetilde{\mathcal{C}}_{\text{str}}(X_H)$. 
Note that $\mathcal{C}_{\text{str}}(X_H)\subset\mathcal{C}(X_H)\subset\mathcal{C}_{\text{loc}}(X_H)$ and $\widetilde{\mathcal{C}}_{\text{str}}(X_H)\subset\widetilde{\mathcal{C}}(X_H)\subset\widetilde{\mathcal{C}}_{\text{loc}}(X_H)$.
\end{definition}

The next lemma is a generalization of Lemma 1.7 in \cite{cq}.

\begin{lemma}
\label{conserved interior}
Fix a Hamiltonian $(n-k)$-form $\alpha\in\Omega^{n-k}_{\text{Ham}}(M)$. If $\alpha$ is a local conserved quantity then $[X_\alpha,X_H]\hk \ \omega=0$, for some (or equivalently every) Hamiltonian multivector field $X_\alpha$ of $\alpha$. Conversely, if $[X_\alpha,X_H]\hk\omega=0$  then $\alpha$ is locally conserved.
\end{lemma}

\begin{proof}
Let $X_\alpha$ be an arbitrary Hamiltonian multivector field of $\alpha$. We have that 
\begin{align*}
[X_\alpha,X_H]\hk\omega&=-X_H\hk d(X_\alpha\hk\omega)-d(X_H\hk X_\alpha\hk\omega)+\\
&\quad{}+X_\alpha\hk(d(X_H\hk\omega))+X_H\hk X_\alpha\hk d\omega &\text{by Prop \ref{interior}}\\
&=- d(X_H\hk X_\alpha\hk\omega)\\
&=-\L_{X_H}(X_\alpha\hk\omega)&\text{by (\ref{Lie})}\\
&=d\L_{X_H}\alpha&\text{by (\ref{dL}).}
\end{align*}
\end{proof}

Recall the following standard result from Hamiltonian mechanics: If $H$ is a Hamiltonian on a symplectic manifold and $f$ and $g$ are two strictly conserved quantities, i.e. $\{f,H\}=0=\{g,H\}$,then $\{f,g\}$ is strictly conserved. This is because $\L_{X_H}\{f,g\}=\{\{f,g\},H\}=0$ by the Jacobi identity.  Moreover, if $f$ and $g$ are local or global conserved quantities (meaning that their bracket with $H$ is constant) then again $\{f,g\}$ is strictly conserved by the Jacobi identity together with the fact that the Poisson bracket with a constant function vanishes. 

The next proposition generalizes these results to multisymplectic geometry. 

\begin{proposition}
\label{Poisson is strictly conserved}
The bracket of two (local, global, or strict) conserved quantities is a strictly conserved quantity. 
\end{proposition}
\begin{proof}
Let $\alpha\in\Omega^{n-k}_{\text{Ham}}(M)$ and $\beta\in\Omega^{n-l}_{\text{Ham}}(M)$ be any two (local, global or strict) conserved quantities. Let $X_\alpha$ and $X_\beta$ denote arbitrary Hamiltonian multivector fields corresponding to $\alpha$ and $\beta$ respectively. By definition,
\begin{align*}
\L_{X_H}\{\alpha,\beta\}&=(-1)^{|\beta|}\L_{X_H}X_\beta\hk X_\alpha\hk\omega.
\end{align*}By (\ref{bracket hook}) together with Lemma \ref{conserved interior} we see that we can commute the Lie derivative and interior product. Hence, \[\L_{X_H}\{\alpha,\beta\}=(-1)^{|\beta|}X_\alpha\hk X_\beta\hk\L_{X_H}\omega.\] The claim now follows since $\L_{X_H}\omega=0$, by Proposition \ref{preserve}. 
\del{
 We have that, 
 \begin{align*}
\L_{X_H}\{\alpha,\beta\}&=d(X_H\hk(\{\alpha,\beta\}))+ X_H\hk X_{\{\alpha,\beta\}}\hk\omega\\
&=(-1)^{|\beta||\alpha|+1}d(X_H\hk\{\beta,\alpha\})+X_H\hk X_{\{\alpha,\beta\}}\hk\omega\\
&=(-1)^{|\beta||\alpha|+1+|\alpha|}d(X_H\hk X_\alpha\hk X_\beta\hk\omega)+X_H\hk X_{\{\alpha,\beta\}}\hk\omega\\
&=(-1)^{|\beta||\alpha|+1+|\alpha|}d(X_H\hk X_\alpha\hk X_\beta\hk\omega)-\{\{\alpha,\beta\},H\}\\
&=(-1)^{|\beta||\alpha|+1+|\alpha|}d(X_H\hk X_\alpha\hk X_\beta\hk\omega)+\{H,\{\alpha,\beta\}\}
\end{align*}
where in the last line we used the graded skew symmetry of the bracket together with $|H|=2$. 
Now we notice that, by proposition \ref{Poisson is Schouten}, we have $\{\{\alpha,H\},\beta\}=(-1)^{|\beta|}X_\beta\hk[X_\alpha,X_H]\hk\omega$, so that by lemma \ref{conserved interior} \[\{\{\alpha,H\},\beta\}=0=\{\{\beta,H\},\alpha\}.\] Hence, by the graded Jacobi identity (proposition \ref{Jacobi}) together with the fact that $|H|=2$,  we obtain \[\{H,\{\alpha,\beta\}\}=(-1)^{|\alpha||\beta|+|\alpha|}d(X_H\hk X_\alpha\hk X_\beta). \] It now follows that \[\L_{X_H}\{\alpha,\beta\}=0\] as desired. 
}
\end{proof}

As a consequence, we obtain:

\begin{proposition}
The spaces $(\widetilde{\mathcal{C}}_{\text{loc}}(X_H),\{\cdot,\cdot\})$, $(\widetilde{\mathcal{C}}(X_H),\{\cdot,\cdot\})$ and $(\widetilde{\mathcal{C}}_{\text{str}}(X_H),\{\cdot,\cdot\})$ are graded Lie subalgebras of $(\widetilde\Omega_{\text{Ham}}(M),\{\cdot,\cdot\})$.
\end{proposition}
\begin{proof}
Proposition \ref{Poisson is strictly conserved} shows that each of these spaces is preserved by the bracket. The claim now follows from Proposition \ref{graded Lie algebra of forms}.
\end{proof}
We conclude this section by showing that the Hamiltonian forms constitute an $L_\infty$-subalgebra of the Lie $n$-algebra of observables. Moreover, restricting a homotopy co-momentum map to the Lie kernel gives an $L_\infty$-morphism into this $L_\infty$-algebra:

Let $\widehat L_\infty(M,\omega)=(\widehat L,\{l_k\})$ denote the graded vector space $\widehat L_i=\Omega^{n-1-i}_{\text{Ham}}(M)$ for $i=0,\ldots, n-1$, together with the maps $l_k$ from the Lie $n$-algebra of observables. 

\begin{theorem}
\label{L infinity subalgebra}
The space $(\widehat L,\{l_k\})$ is an $L_\infty$-subalgebra of $(L,\{l_k\})$.
\end{theorem}

\begin{proof}
We note that $l_1$ preserves $\widehat L$ since closed forms are Hamiltonian. For $k>1$, since $l_k$ vanishes on elements of positive degree we need only consider $l_k(\alpha_1,\cdots,\alpha_k)=-(-1)^{\frac{k(k+1)}{2}}X_{\alpha_k}\hk\cdots\hk X_{\alpha_1}\hk\omega$, where $\alpha_1,\ldots,\alpha_k$ are Hamiltonian $(n-1)$-forms. By Lemma \ref{rogers 3.7} we see that $l_k(\alpha_1,\ldots,\alpha_k)$ is a Hamiltonian $(n+1-k)$-form.
\end{proof}

\begin{proposition}
\label{conserved quantities are L infinity}
The spaces $\mathcal{C}(X_H)\cap \widehat L$ , $\mathcal{C}_{\text{loc}}(X_H)\cap\widehat L$ and $\mathcal{C}_{\text{str}}(X_H)\cap\widehat L$ are $L_\infty$-subalgebras of $\widehat L_\infty(M,\omega)$.
\end{proposition}

\begin{proof}

The proof is analogous to the proof of Proposition 1.15 in \cite{cq}. Since the proof is short, we include it here. From Theorem \ref{L infinity subalgebra} we see that each of the spaces $\mathcal{C}(X_H)\cap \widehat L$ , $\mathcal{C}_{\text{loc}}(X_H)\cap\widehat L$ and $\mathcal{C}_{\text{str}}(X_H)\cap\widehat L$ are closed under each $l_k$. It remains to show that for Hamiltonian $(n-1)$-forms $\alpha_1,\cdots,\alpha_k$ which are (locally, globally, strictly) conserved, that $l_k(\alpha_1,\cdots,\alpha_k)$ is (locally, globally, strictly) conserved. Indeed, \[\L_{X_H}l_k(\alpha_1,\cdots,\alpha_k)=\L_{X_H}X_{\alpha_k}\hk\cdots\hk X_{\alpha_1}\hk\omega.\] Using equation (\ref{bracket hook}) together with Lemma \ref{conserved interior} we see that we can commute the Lie derivative and interior product. The claim then follows since $\L_{X_H}\omega=0$.
\end{proof}

\subsection{Continuous Symmetries and their Algebraic Structure}

Fix a multi-Hamiltonian system $(M,\omega,H)$. Our motivation for the definition of a continuous symmetry comes from Hamiltonian mechanics; we directly generalize the definition. As is the case with conserved quantities, we define three types of continuous symmetry. 

\begin{definition}
We say that a Hamiltonian multivector field $X\in\X_{\text{Ham}}(M)$ is 

\begin{itemize}
\item a local continuous symmetry if $\L_XH$ is closed,
\item a global continuous symmetry if $\L_XH$ is exact,
\item a strict continuous symmetry if $\L_XH=0$.
\end{itemize}
Note that a continuous symmetry automatically preserves $\omega$ by Proposition \ref{preserve}. We denote the space of local, global, and strict continuous symmetries by $\mathcal{S}_{\text{loc}}(H)$, $\mathcal{S}(H)$, and $\mathcal{S}_{\text{str}}(H)$ respectively. Moreover, we let $\widetilde{\mathcal{S}}_{\text{loc}}(H)$, $\widetilde{\mathcal{S}}(H)$, and $\widetilde{\mathcal{S}}_{\text{str}}(H)$ denote the quotient by the kernel of $\omega$. %Note that, by definition,  $\widetilde{\mathcal{S}}_{\text{loc}}(H)=\widetilde{\mathcal{S}}(H)=\widetilde{\mathcal{S}}_{\text{str}}(H)$.

We will say that a multivector field $X$ is a weak (local, global, strict) continuous symmetry if $\L_X\omega=0$ and $\L_XH$ is closed, exact, or zero respectively. That is, a weak continuous symmetry is not necessarily Hamiltonian.
\end{definition}

\del{

\begin{remark}
Motivation for calling these things symmetries comes from the fact that with this definition of symmetry sets up a correspondence with our notion of conserved quantity. 
\end{remark}

For a Lie group acting on a symplectic Hamiltonian system, we call each infinitesimal generator a symmetry. For example, on the standard phase space $(T^\ast\R^3,\omega)$ with Hamiltonian $H=K-U$, we would call $SO(3)$ a symmetry group since the flow of each infinitesimal generator preserves $H$. In the multisymplectic setup, $H$ is replaced with a higher degree form and the infinitesimal generators are replaced with the wedge products of infinitesimal generators.   

\begin{example}
Consider $\R^3$ with $2$-plectic form $\mu=\text{vol}$. Let $H=E$ be the electromagnetic field. Want charge to be conserved so find that symmetry.  We will see later that under Noether's theorem, the corresponding conserved quantity is charge.
\end{example}

\begin{remark} For the case of the multisymplectic manifold $(\R^n,\text{vol})$,  in sectino ?? we give a geometric interpretation of a symmetry. In particular, we will see that the symmetries of degree $k$ are in one-to-one correspondence to $k$-Lagrangian submanifolds of $\R^{2k}$.
\end{remark}
}

\del{
Proposition \ref{kernel} showed that any two Hamiltonian vector fields corresponding to a Hamiltonian form differ by something in the kernel of $\omega$. Conversely, we have that 

\begin{proposition}
If two symmetries differ by something in the kernel of $\omega$, then any of the conserved quantities they generate differ by a closed form. 
\end{proposition}

\begin{proof}
Suppose that $X$ and $Y$ are (local, global or strict) symmetries with arbitrary corresponding conserved quantities $\alpha$ and $\beta$ respectively.  Suppose that $X-Y$ is in the kernel of $\omega$. Then \[d\alpha=X\hk \omega=Y\hk\omega=d\beta\] and so $\alpha$ and $\beta$ differ by a closed term.
\end{proof}

}

\begin{proposition}
\label{graded Lie algebra of vector fields}
We have that $(\X_{\text{Ham}}(M),[\cdot,\cdot])$ is a graded Lie subalgebra of  $(\Gamma(\Lambda^\bullet(TM)),[\cdot,\cdot])$. 
\end{proposition}
\begin{proof}
By equation (\ref{interior equation}) we see that $[X,Y]\hk\omega=(-1)^ld(X\hk Y\hk\omega)$. Hence the space of Hamiltonian multivector fields is closed under the Schouten bracket. 
\end{proof}

\begin{proposition}
\label{symmetry super algebra}
The spaces $\mathcal{S}_{\text{loc}}(H)$, $\mathcal{S}(H)$, and $\mathcal{S}_{\text{str}}(H)$ are graded Lie subalgebras of $(\X_{\text{Ham}}(M),[\cdot,\cdot])$.
\end{proposition}

\begin{proof}
We see that each of $\mathcal{S}_{\text{loc}}(H)$, $\mathcal{S}(H)$, and $\mathcal{S}_{\text{str}}(H)$ are closed under the Schouten bracket directly from equations (\ref{dL}) and (\ref{L bracket}).
\end{proof}
\del{
The next example shows that none of $\mathcal{S}_{\text{loc}}(H), \mathcal{S}(H)$ or $\mathcal{S}_{\text{str}}(H)$ are necessarily closed under the wedge product.
\begin{example}
To Do.
\end{example}
}
The next lemma generalizes Lemma 2.9 (ii) of \cite{cq}.
\begin{lemma}
\label{symmetry interior}
Let $Y\in\Gamma(\Lambda^k(TM))$. If $Y$ is a local symmetry, then $[Y,X_H] \hk\  \omega=0$. Conversely, if $[Y,X_H]\hk\ \omega=0$ and $\L_Y\omega=0$, then $Y$ is a local symmetry.
\end{lemma}

\begin{proof}
We have that 
\begin{align*}[Y,X_H]\hk\omega&=(-1)^{k+1}\L_Y(X_H\hk\omega)-X_H\hk\L_Y\omega&\text{by (\ref{bracket hook})}\\
&=(-1)^{k}\L_YdH&\text{since $\L_Y\omega=0$ and $X_H\hk\omega=-dH$}\\
&=-d\L_YH&\text{by (\ref{dL}).}
\end{align*}
\end{proof}

Recall that for a group $G$ acting on a manifold $M$ we had defined in equation $(\ref{S_k})$ the set $S_k:=\{V_p \ ; \ p\in\Rho_{\g,k}\}$. Proposition \ref{infinitesimal generator of Schouten} showed that $S=\oplus S_k$ was a graded Lie algebra.  We now get the following.

\begin{proposition}
\label{symmetries are L infinity}The spaces  $\mathcal{S}_{\text{loc}}(H)\cap S$, $\mathcal{S}(H)\cap S$, and $\mathcal{S}_{\text{str}}(H)\cap S$ are graded Lie subalgebras of $S$. 
\end{proposition}
\begin{proof}
By Proposition \ref{symmetry super algebra} we have that the spaces of symmetries are preserved by the Schouten bracket. The claim now follows by Proposition \ref{infinitesimal generator of Schouten}.
\end{proof}

\del{

\begin{corollary}
If $Y$ is a (local,global, strict) symmetry, then $Y\hk\omega$ is a strict conserved quantity.
\end{corollary}

\begin{proof}e roles of $Y$ and $X_H$ in the above lemma, we have that \[[X_H,Y]\hk\omega=\L_{X_H}(Y\hk\omega).\]
\end{proof}

}
 \section{Noether's Theorem in Multisymplectic Geometry}
In this section we show how Noether's theorem extends from symplectic to multisymplectic geometry. To see this generalization explicitly, we first recall how Noether's theorem works in symplectic geometry.

\subsection{Noether's Theorem in Symplectic Geometry}

Let $(M,\omega,H)$ be a Hamiltonian system. That is $(M,\omega)$ is symplectic and $H$ is in $C^\infty(M)$. Noether's theorem gives a correspondence between symmetries and conserved quantities. If $f\in C^\infty(M)$ is a (local, global) conserved quantity then $X_f$ is a (local, global) continuous symmetry. Conversely, if a vector field $X_f$ is a (local, global) continuous symmetry, then $f$ is a (local, global) conserved quantity. Note that in the symplectic case, local and strict symmetries and conserved quantities are the same thing.

If $X$ is only a weak (local, global) continuous symmetry, then $\L_X\omega=0$ so that by the Cartan formula  around each point there is a neighbourhood $U$ and a function $f\in C^\infty(U)$ such that $X=X_f$ on $U$.  This function $f$ is a (local, global) conserved quantity in the Hamiltonian system $(U,\omega|_U,H|_U)$.

If we only consider the symmetries and conserved quantities coming from a co-momentum map $\mu:\g\to C^\infty(M)$ then, under the assumption of an $H$-preserving group action, each symmetry $\xi$ has corresponding global conserved quantity $\mu(\xi)$ and vice versa. 

The rest of this subsection formalizes this, and the following sections will generalize it to multi-symplectic geometry.

Recall that an equivariant co-momentum map gives a Lie algebra morphism between $(\g,[\cdot,\cdot])$ and $(C^\infty(M),\{\cdot,\cdot\})$.

\begin{proposition}
\label{difference}
Let $\mu:\g\to C^\infty(M)$ be a momentum map. For $\xi,\eta\in\g$ we have that $\mu([\xi,\eta])=\{\mu(\xi),\mu(\eta)\}+\text{constant}$. If the moment map is equivariant then  $\mu([\xi,\eta])=\{\mu(\xi),\mu(\eta)\}$.
\end{proposition}

\begin{proof}
See Theorem 4.2.8 of \cite{Marsden}.
\end{proof}

\begin{remark}
The constant in the above proposition actually has a specific form. It is given by a specific Lie-algebra cocycle. It turns out that this cocycle has a generalization to multisymplectic geometry, a topic being explored by the author in \cite{future}.
\end{remark}

\del{

First note that if two conserved quantities differ by a constant they give the same symmetry. Moreover if two symmetries differ by something in the kernel of $\omega$ then they give the same conserved quantity. These two results are trivial, but we state them as propositions for reference when we generalize to the multisymplectic case.

\begin{proposition}
\label{difference of cq}
If $f,g\in C^\infty(M)$ are conserved quantities and $f=g+c$ then $X_f=X_g$. 
\end{proposition}

\begin{proposition}
\label{difference of cs}
If $X$ and $Y$ are symmetries whose difference is in the kernel of $\omega$, then they generate conserved quantities that differ by a constant.
\end{proposition}
}
As stated above, it is clear that in the symplectic case $\mathcal{C}_{loc}(X_H)=\mathcal{C}_{str}(X_H)$ and $\mathcal{S}_{loc}(H)=\mathcal{S}_{str}(H)$. It is easily verified that the map $\alpha\mapsto X_\alpha$ is a Lie algebra morphism from  $(\mathcal{C}(X_H),\{\cdot,\cdot\})$ to $(\mathcal{S}(H),[\cdot,\cdot])$ and from $(\mathcal{C}_{loc}(X_H),\{\cdot,\cdot\})$ to $(\mathcal{S}_{loc}(H),[\cdot,\cdot])$ . However, under the quotients this map turns into a Lie algebra isomorphism. 

\begin{proposition}
\label{graded Lie iso}
The map $\alpha \mapsto X_\alpha$ is a Lie algebra isomorphism from $(\widetilde{C}(X_H),\{\cdot,\cdot\})$ to $(\widetilde{S}(H),[\cdot,\cdot])$ and $(\widetilde{\mathcal{C}}_{loc}(X_H),\{\cdot,\cdot\})$ to $(\widetilde{\mathcal{S}}_{loc}(H),[\cdot,\cdot])$ .\end{proposition}

As a consequence of this proposition, we can now see how a momentum map sets up a Lie algebra isomorphism between the symmetries and conserved quantities it generates. Let $C=\{\mu(\xi); \xi\in\g\}$ and $S=\{V_\xi;\xi\in\g\}$. Let $\widetilde C$ be the quotient of $C$ by constant functions. Let $\widetilde S$ denote the quotient of $S$ by the kernel of $\omega$. Since the kernel of $\omega$ is trivial, $S=\widetilde{S}$. Then we get an induced well defined Poisson bracket on $\widetilde C$ and an induced well defined Lie bracket on $\widetilde S$.  We thus get a Lie algebra isomorphism:

\begin{proposition}
\label{la isomorphism}
The map between $(\widetilde C,\{\cdot,\cdot\})$ and $(\widetilde S,[\cdot,\cdot])$ that sends $[V_\xi]$ to $[\mu(\xi)]$ is a Lie-algebra isomorphism.

\end{proposition}

With our newly defined notions of symmetry and conserved quantity on a multisympletic manifold, we now exhibit how these concepts generalize to the setup of multisymplectic geometry. 

\del{

If two conserved quantties differ by a constant, then they will generate the same symmetry. Conversely, if two symmetries differ by something in the kernel of $\omega$ (which is zero in the symplectic case) then they will generate the same conserved quantity.co-moment

A moment map gives us symmetries, each $\xi$ generates a symmetry and by Noether we see that the corresponding conserved quantity is $\mu(\xi)$. We know that the moment map is a Lie algebra homomorphism between. (123)

Consider symmetries of the form. Consider their quotient by constant functions. Consider conserved quantites of form...and their quotient by kernel of $\omega$. Then the Lie bracket/Poisson descends to well defined bracket. The moment map gives us the following Lie algebra isomorphism.  (1234)

}
\subsection{The Correspondence between Mutlisymplectic Conserved Quantities and Continuous Symmetries}

We first examine the correspondence between symmetries and conserved quantities on multi-Hamiltonian systems. We will make repeated use of the following equations. Fix $\alpha\in\Omega^{n-k}_{\mathrm{Ham}}(M)$. By definition we have that \[\{\alpha,H\}=-X_H\hk X_\alpha\hk\omega=X_{H}\hk d\alpha=\L_{X_H}\alpha-d(X_H\hk\alpha).\] But we also know that $\{\alpha,H\}=-\{H,\alpha\}$, since $|H|=2$. Thus, by definition of the Poisson bracket and equation (\ref{Lie}) we have that \[-\{H,\alpha\}=(-1)^kX_\alpha\hk X_H\hk\omega=(-1)^{k+1}X_\alpha\hk dH=\L_{X_\alpha}H-d(X_\alpha\hk H).\]Putting these together we obtain 
\begin{equation}
\label{Noether 1}
\L_{X_\alpha}H=d(X_\alpha\hk H)+\L_{X_H}\alpha-d(X_H\hk\alpha)
\end{equation}
and
\begin{equation}
\label{Noether 2}
\L_{X_H}\alpha=d(X_H\hk\alpha)+\L_{X_\alpha}H-d(X_\alpha\hk H).
\end{equation}
\del{
As in the case of symplectic geometry, a symmetry only gives a conserved quantity locally, and this will be resolved when we consider symmetries and conserved quantities coming from a homotopy co-momentum map.}
\begin{theorem} \label{Noether theorem 1}If $\alpha\in\Omega^{n-k}_{\text{Ham}}(M)$ is a (local, global) conserved quantity then any corresponding Hamiltonian $k$-vector field is a (local, global) continuous symmetry. Conversely, if $A\in\Gamma(\Lambda^k(TM))$ is a (local, global) continuous symmetry, then any corresponding Hamiltonian form is a (local,global) conserved quantity.

\end{theorem}

\begin{proof} 
Consider $\alpha\in\Omega^{n-k}_{\text{Ham}}(M)$. Let $X_\alpha$ be an arbitrary Hamiltonian multivector field.  Then, by equation (\ref{Noether 1}) we have that
\[\L_{X_\alpha}H=d(X_\alpha\hk H)+\L_{X_H}\alpha-d(X_H\hk\alpha).\]

Thus, if $\alpha$ is a (local or global) conserved quantity then $X_\alpha$ is a (local or global) continuous symmetry.

Conversely, suppose that $A$ is a (local or global) continuous symmetry and let $\alpha$ be a corresponding Hamiltonian form. Following the same argument above, we have by equation (\ref{Noether 2})
\[\L_{X_H}\alpha=d(X_H\hk\alpha)+\L_{X_\alpha}H-d(X_\alpha\hk H)\]
\del{
In summary we have the following correspondences: \[\mathcal{S}_{\text{loc}}(H) \longleftrightarrow \mathcal{C}_{\text{loc}}(X_H) \ \ \ \ \ \ \ \ \ \ \mathcal{S}(H) \longleftrightarrow \mathcal{C}(X_H)\]and \[\mathcal{S}_{\text{str}}(H) \longleftrightarrow \mathcal{C}(X_H) \ \ \ \ \ \ \ \ \ \ \mathcal{C}_{\text{str}}(X_H) \longleftrightarrow \mathcal{S}(H)\]
 }
\end{proof}

The correspondence between strictly conserved quantities and strict continuous symmetries is a little bit different. We have that

\begin{corollary}
If $\alpha\in\Omega^{n-k}_{\mathrm{Ham}}(M)$ is a strictly conserved quantity then $X_\alpha$ is a global continuous symmetry. Conversely, if $A$ is a strict continuous symmetry then the corresponding Hamiltonian $(n-k)$-form $\alpha$ is a global conserved quantity. 

\end{corollary}

\begin{proof}
This follows from the proof of the above theorem.
\end{proof}

\begin{remark}
If we were to consider weak continuous symmetries in the above theorem, then by the Poincar\'e lemma, a continuous symmetry would still give a conserved quantity, but only in a neighbourhood around each point of the manifold.
\end{remark}

The following simple example exhibits the correspondence.
\begin{example}
Consider $M=\R^3$ with volume form $\omega=dx\wedge dy\wedge dz$, $H=-xdy$ and $\alpha=zdx$. Then $dH=-dx\wedge dy$ so that $X_H=\frac{\pd}{\pd z}$. Also, $d\alpha=dz\wedge dx$ and so $X_\alpha=\frac{\pd}{\pd y}$. By the Cartan formula, we have that \[\L_{X_\alpha}H=-dx +dx=0,\] which means that $X_\alpha\in\mathcal{S}_{\text{str}}(H)\subset\mathcal{S}(H)$. We also have that \[\L_{X_H}\alpha=d(X_H\hk \alpha)+\{\alpha,H\}=d(X_H\hk\alpha)-d(X_\alpha\hk H)=dx.\] That is, $\alpha\in \mathcal{C}(X_H)$. Thus $\alpha$ is a global conserved quantity and $X_\alpha$ is a global continuous symmetry.
\end{example}

\subsection{A Homotopy Co-Momentum  Map as a Morphism}

We work with a fixed multi-Hamiltonian system $(M,\omega,H)$ with acting symmetry group $G$. By definition, a co-momentum map is an $L_\infty$-morphism between the Chevalley-Eilenberg complex and the Lie $n$-algebra of observables. Recall that in Section 4 we had defined the $L_\infty$-algebra $\widehat L(M,\omega)$, where $\widehat L$ consisted entirely of Hamiltonian forms: $\widehat L=\oplus_{i=0}^{n-1}\Omega^{n-1-i}_{\text{Ham}}(M)$.  

\begin{proposition}
A weak homotopy co-momentum map is an $L_\infty$-morphism from $(\Rho_\g,\partial,[\cdot,\cdot])$ to $(\widehat L, \{l_k\})$.
\end{proposition}

\begin{proof}
Equation (\ref{hcmm kernel}) shows that a homotopy co-momentum map sends each element of the Lie kernel to a Hamiltonian form. Hence the claim follows from Proposition \ref{Schouten is closed} and Theorem \ref{L infinity subalgebra}.
\end{proof}

Next we study how a weak homotopy co-momentum  map interacts with the generalized Poisson bracket on the space of Hamiltonian forms. In particular, to make a connection with Proposition \ref{difference} from symplectic geometry, we compare the difference of $f_{k+l-1}([p,q])$ and $\{f_k(p),f_l(q)\}$.

Let $G$ be a Lie group acting on a multi-Hamiltonian system $(M,\omega,H)$. Let $(f)$ be a weak homotopy co-momentum  map. By equation (\ref{hcmm kernel}) we see that under this restriction the image of the co-momentum map is contained in the $L_\infty$-algebra  $\widehat L(M,\omega)$ of Hamiltonian forms. Moreover, we obtain that every element in the image is a conserved quantity. This was one of the main points of \cite{cq}. Indeed, in \cite{cq} Propositions 2.12 and 2.21 say:

\begin{proposition}
\label{H action}
If the group locally or globally  preserves $H$, then $f_k(p)$ is a local conserved quantity for all $p\in\Rho_{\g,k}$. If the group strictly preserves $H$ then $f_k(p)$ is a globally conserved quantity for all $p\in\Rho_{\g,k}$.
\end{proposition}

Thus, by restricting a homotopy co-momentum  map to the Lie kernel, we see that, under assumptions on the group action, every element is a conserved quantity, analogous to the setup in symplectic geometry.

As a consequence of Theorem \ref{Noether theorem 1} we see that the moment map also gives a family of continuous symmetries.

\begin{proposition}\label{H action 2}
If the group locally or globally  preserves $H$, then $V_p$ is a local continuous symmetry for all $p\in\Rho_{\g,k}$. If the group strictly preserves $H$ then $V_p$ is a global continuous symmetry for all $p\in\Rho_{\g,k}$.
\end{proposition}
\begin{example}\bf{(Motion in a conservative system under translation)}\rm

Recall that in Example \ref{translation} we considered the translation action of $\R^3$ on $(M,\omega,H)$ where $M=T^\ast\R^3=\R^6$, $\omega=\text{vol}$ and \[H=\frac{1}{2}\left((p_1q^2dq^3-p_1q^3dq^2)-(p_2q^1dq^3-p_2q^3dq^2)+(p_3q^1dq^2-p_3q^2dq^1)\right)dp_1dp_2dp_3,\] where $q^1$, $q^2$, $q^3$ are the standard coordinates on $\R^3$ and $q^1$, $q^2$, $q^3$, $p_1$, $p_2$, $p_3$ are the induced coordinates on $T^\ast\R^3$. It is easy to check that each of $\L_{\frac{\pd}{\pd q^i}}H$ are exact for $i=1,2,3$. That is,  the group action globally preserves $H$. Hence, by Proposition \ref{H action} each of the differential forms, computed in Example \ref{translation}, \[f_1(e_1)=\frac{1}{2}(q^2dq^3-q^3dq^2)dp_1dp_2dp_3, \] \[ f_1(e_2)=\frac{1}{2}(q^1dq^3-q^3dq^1)dp_1dp_2dp_3, \] \[f_1(e_3)=\frac{1}{2}(q^1dq^2-q^2dq^2)dp_1dp_2dp_3,\]
\[f_2(e_1\wedge e_2)=q^3dp_1dp_2dp_3, \ \  \ \ \ \ f_2(e_1\wedge e_3)=q^2dp_1dp_2dp_3, \ \ \ \ \ \ f_2(e_2\wedge e_3)=q^1dp_1dp_2dp_3,\]and
\[f_3(e_1\wedge e_2\wedge e_3)=\frac{1}{3}\left(p_1dp_2dp_3+p_2dp_3dp_1+p_3dp_1dp_2\right).\]
are all globally conserved.  Thus, by Example \ref{translation}, the Lie derivative of these differential forms by the geodesic spray are all exact.

Moreover, by Proposition \ref{H action 2}, each of $\frac{\pd}{\pd q^1},\frac{\pd}{\pd q^2},\frac{\pd}{\pd q^3},\frac{\pd}{\pd q^1}\wedge\frac{\pd}{\pd q^2},\frac{\pd}{\pd q^1}\wedge\frac{\pd}{\pd q^3}\frac{\pd}{\pd q^3},\frac{\pd}{\pd q^2}\wedge\frac{\pd}{\pd q^3}$ and $\frac{\pd}{\pd q^1}\wedge\frac{\pd}{\pd q^2}\wedge\frac{\pd}{\pd q^3}$ are global continuous symmetries in this multi-Hamiltonian system. 

\end{example}

\begin{proposition}
\label{Schouten under moment}Let $p\in\Rho_{\g,k}$ and $q\in\Rho_{\g,l}$ be arbitrary. Set $\alpha=f_{n-k}(p)$ and $\beta=f_{n-l}(q)$. Then we have that $-\zeta(k)\zeta(l)[V_p,V_q]$ is a Hamiltonian multivector field for $\{\alpha,\beta\}$.
\end{proposition}

\begin{proof}
By definition of the co-momentum map (equation (\ref{hcmm kernel})) we have that $X_\alpha-\zeta(k)V_p$ and $X_\beta-\zeta(l)V_q$ are in the kernel of $\omega$. Hence, by Proposition \ref{well defined} we have that \[[X_\alpha-\zeta(k)V_p,X_\beta-\zeta(l)V_q]\hk\omega=0.\]Proposition \ref{well defined} also shows that \[[X_\alpha-\zeta(k)V_p,X_\beta-\zeta(l)V_q]\hk\omega=([X_\alpha,X_\beta]+\zeta(k)\zeta(l)[V_p,V_q])\hk\omega.\]Thus \[[X_\alpha,X_\beta]\hk\omega=-\zeta(k)\zeta(l)[V_p,V_q]\hk\omega.\]

The claim now follows from Proposition \ref{Poisson is Schouten}.
\end{proof}

Our generalization of Proposition \ref{difference} to multisymplectic geometry is:
\begin{proposition}
\label{difference is closed}
For $p\in\Rho_{\g,k}$ and $q\in\Rho_{\g,l}$ we have that $\{f_k(p),f_l(q)\}-(-1)^{k+l+kl}f_{k+l-1}([p,q])$ is a closed $(n+1-k-l)$-form.
\end{proposition}
\begin{proof}
By definition of co-momentum map (equation (\ref{hcmm})) we have that 

\begin{align*}
d(f_{k+l-1}([p,q]))&=-f_{k+l-2}(\partial[p,q])-\zeta(k+l-1)V_{[p,q]}\hk\omega\\
&=-\zeta(k+l-1)V_{[p,q]}\hk\omega&\text{by Proposition \ref{Schouten is closed}}\\
&=\zeta(k+l-1)[V_p,V_q]\hk\omega&\text{by Proposition \ref{infinitesimal generator of Schouten}}\\
&=\zeta(k+l-1)\zeta(k)\zeta(l)[X_{f_k(p)},X_{f_l(p)}]\hk\omega&\text{by Proposition \ref{Schouten under moment}}\\
&=-(-1)^{k+l+kl}X_{\{f_k(p),f_l(q)\}}\hk\omega&\text{by Proposition \ref{Poisson is Schouten} and Remark \ref{ugly signs}}\\
&=(-1)^{k+l+kl}d(\{f_k(p),f_l(q)\})&\text{by definition.}
\end{align*}
\end{proof}

\begin{remark}
In the symplectic case we have that $\mu([\xi,\eta])-\{\mu(\xi),\mu(\eta)\}$ is closed (i.e. constant) and if $\mu$ is equivariant then this constant is zero. In the multiysymplectic set up, the above proposition shows that $f_{k+l-1}([p,q])-\{f_k(p),f_l(q)\}$ is a closed form. It can be shown that if the co-momentum map $(f)$, restricted to the Lie kernel, is equivariant, then this difference is exact (this is the content of ongoing research in \cite{future}). We thus have obtained the following generalization from symplectic geometry: A homotopy co-momentum map, restricted to the Lie kernel, is equivariant if and only if $\{f_k(p),f_l(q)\}=f_{k+l-1}([p,q])$ in deRham cohomology.
\end{remark}

From this proposition we see that a co-momentum map does not necessarily preserve brackets; however, we now show that once we pass to certain cohomology groups then it will. Moreover, the co-momentum map will give an isomorphism of graded Lie algebras, generalizing Proposition \ref{la isomorphism}. Recall that we had defined  $\widetilde\X^k_{\text{Ham}}(M)$ to be the quotient of $\X^k_{\text{Ham}}(M)$ by the kernel of $\omega$ restricted to $\Lambda^k(TM)$. We set $\widetilde\X_{\text{Ham}}(M)=\oplus\widetilde\X^k_{\text{Ham}}(M)$. 

\begin{proposition}
The Schouten bracket on $\X_{\mathrm{Ham}}(M)$ descends to a well defined bracket on $\widetilde \X_{\mathrm{Ham}}(M)$.

\end{proposition}

\begin{proof} This follows directly from Proposition \ref{well defined}.
\end{proof}
Similarily, we let $\widetilde\Omega^{n-k}_{\text{Ham}}(M)$ denote the quotient of $\Omega^{n-k}_{\text{Ham}}(M)$ by the closed forms of degree $k$ and set $\widetilde\Omega_{\text{Ham}}(M)=\oplus\Omega^{n-k}_{\text{Ham}}(M)$. Recall that Proposition \ref{graded Lie algebra of forms} showed that $(\widetilde\Omega_{\text{Ham}}(M),\{\cdot,\cdot\})$ was a well defined graded Lie algebra. 
\begin{theorem}
\label{la iso}
The map $\alpha\mapsto X_\alpha$ is an isomorphism of graded Lie algebras from $(\widetilde\Omega_{\mathrm{Ham}}(M),\{\cdot,\cdot\})$ to $(\widetilde\X_{\mathrm{Ham}}(M),[\cdot,\cdot])$.
\end{theorem}

\begin{proof}
The map is well defined since the Hamiltonian multivector field of a closed form is the zero vector field.  The map is clearly surjective. It is injective since if $X_\alpha=X_\beta$ then $d\alpha=d\beta$. Lastly, by Lemma \ref{Poisson is Schouten}, we have that $X_{\{\alpha,\beta\}}=[X_\alpha,X_\beta]$ in the quotient space.
\end{proof}
We have now obtained a generalization of Proposition \ref{graded Lie iso} from symplectic geometry.
\begin{corollary} The map $\alpha\mapsto X_\alpha$ is a graded Lie algebra isomorphism from $(\widetilde{\mathcal{C}}(X_H),\{\cdot,\cdot\})$ to $(\widetilde{\mathcal{S}}(H),[\cdot,\cdot])$ and from $(\widetilde{\mathcal{C}}_{loc}(X_H),\{\cdot,\cdot\})$ to $(\widetilde{\mathcal{S}}_{loc}(H),[\cdot,\cdot])$.

\end{corollary}
\begin{proof}
We know from Proposition \ref{Poisson is strictly conserved} that each of the spaces of conserved quantities are closed under the Poisson bracket. Similarly, by Proposition \ref{symmetry super algebra}, the spaces of continuous symmetries are all closed under the Schouten bracket. The claim now follows from Theorem \ref{la iso}.
\end{proof}

We now give a generalization of Proposition \ref{la isomorphism} to multisymplectic geometry: We let $C_k$ denote the image of the Lie kernel under the co-momentum map. That is, let $C_k=f_{n-k}(\Rho_{\g,n-k})$. Let $\widetilde C_k$ denote the quotient of $C_k$ by closed forms and set $\widetilde C=\oplus \widetilde C_k$.   Recall that we had defined $S_k$ to be the set $\{V_p; p\in\Rho_{\g,k}\}$. Let $\widetilde S_k$ denote the quotient of $S_k$ by elements in the Lie kernel and set $\widetilde S=\oplus\widetilde S_k$.  Our generalization of Proposition \ref{la isomorphism} is given by the following corollaries.

\begin{corollary}
A momentum map induces an $L_\infty$-algebra morphism from $\widetilde S$ to $\widetilde C \cap \widehat L$ given by $V_p\mapsto f_k(p)$.
\end{corollary}

\begin{proof}
We see by Proposition \ref{difference is closed} that the Poisson bracket preserves $\widetilde C$. The claim follows since by definition a homotopy co-momentum map is an $L_\infty$-morphism.
\end{proof}

\begin{corollary}
For a group action that is (locally, globally or strictly) $H$-preserving, an equivariant homotopy co-momentum map induces an isomorphism of graded Lie algebras between $(\widetilde S,[\cdot,\cdot])$ and $(\widetilde C, \{\cdot,\cdot\})$. Explicitly, the map is given by $[V_p]\mapsto [f_k(p)]$. 
\end{corollary}

\begin{proof}
The Lie algebra isomorphism given in Theorem \ref{la iso} is precisely the co-momentum map. Indeed, if $\alpha=f_k(p)$ for $p\in\Rho_{\g,k}$, then $X_{f_k(p)}=V_p$ in $\widetilde\X_{\text{Ham}}(M)$, since both are Hamiltonian vector fields for $\alpha$. Proposition \ref{difference is closed} now shows that the co-momentum map preserves the Lie brackets on these quotient spaces.

\end{proof}

\del{

Now by proposition \ref{H action}, we know that the image of the moment map, when restricted to the Lie kernel, contains entirely conserved quantities, if the action preserves $H$.  Indeed, proposition \ref{H action} says that if the action locally or globally preserves $H$ then $C$ is contained in $C_{\text{loc}}(X_H)$ and if the action strictly preserves $H$ then $C$ is  contained in $C(X_H)$.

\begin{corollary} For a locally or globally $H$ preserving group action, an equivariant homotopy comomentum moment map gives a graded Lie algebra isomorphism from $(C_{\text{loc}}(X_H)\cap \widetilde C,\{\cdot,\cdot\})$ to $(S_{\text{loc}}(H)\cap\widetilde S,[\cdot,\cdot])$. If the action strictly preserves $H$, then the comomentum map gives a graded Lie algebra isomorphism from $(C(X_H)\cap \widetilde C,\{\cdot,\cdot\})$ to $(S(H)\cap\widetilde S,[\cdot,\cdot])$
\end{corollary}

}

\section{Applications}
We first apply the generalized Poisson bracket to extend the theory of classical momentum and position functions on the phase space of a manifold to the multisymplectic phase space.
\subsection{Classical Multisymplectic Momentum and Position Forms}
Recall the following notions from Hamiltonian mechanics:

Let $N$ be a manifold and $(T^\ast N,\omega=-d\theta)$ the canonical phase space. Given a group action on $N$ we can extend this to a group action on $T^\ast N$ that preserves both the tautological forms $\theta$ and $\omega$. It is easy to check that a co-momentum map for this action is given by $f:\g\to C^\infty(N) \ , \ \xi\mapsto V_\xi\hk\theta$. From this co-momentum map we can introduce the classical momentum functions, as discussed in Propositions 4.2.12 and 5.4.4 of \cite{Marsden}. Given $X\in\Gamma(TN)$, its classical momentum function is $P(X)\in C^\infty(T^\ast N)$ defined by $P(X)(\alpha_q):= \alpha_q(X_q)$. Corollary 4.2.11 of \cite{Marsden} then shows that \begin{equation}\label{momentum form equal}P(V_\xi)=f(\xi).\end{equation} Next, given $h\in C^\infty(N)$ define $\widetilde h\in C^\infty(T^\ast N)$ by $\widetilde h= h\circ\pi$. The function $\widetilde h$ is referred to as the corresponding position function. The following Poisson bracket relations between the momentum and position functions are then obtained in Proposition 4.2.12 of \cite{Marsden}: 
\begin{equation}\label{classical 1}\{P(X),P(Y)\}=P([X,Y])\end{equation}
\begin{equation}\label{classical 2}\{\widetilde h,\widetilde g\}=0\end{equation}and
\begin{equation}\label{classical 3}\{\widetilde h, P(X)\}=\widetilde{X(h)}.\end{equation} 
These bracket relations are the starting point for obtaining a quantum system from a classical system. 

\begin{remark}In \cite{Marsden} their first bracket relation actually reads $\{P(X),P(Y)\}=-P([X,Y])$. This is because their defining equation for a Hamiltonian vector field is $dh=X_h\hk\omega$, as compared to our $dh=-X_h\hk\omega$. 
\end{remark}

The goal of this subsection is to show how these concepts generalize to the multisymplectic phase space. As in Example \ref{Multisymplectic Phase Space}, let $N$ be a manifold and $(M,\omega)$ the multisymplectic phase space. That is, $M=\Lambda^k(T^\ast N)$ and $\omega=-d\theta$ is the canonical $(k+1)$-form on $M$. Let $\pi:M\to N$ denote the projection map. In Example \ref{multisymplectic phase space} we showed that \[f_l:\Rho_{\g,l}\to\Omega^{k-l}_{\text{Ham}}(M) \ \ \ \ \ \ \ \ \ \ p\mapsto -\zeta(l+1)V_p\hk\theta\] was a weak homotopy co-momentum map for the action on $M$ induced from the action on $N$.

\begin{definition}
Given decomposable $X=X_1\wedge\cdots\wedge X_l$ in $\Gamma(\Lambda^l(TN))$ we define its momentum form $P(X)\in \Omega^{k-l}(M)$ by \[P(X)(\mu_x)(Z_1,\dots, Z_{k-l}):=-\zeta(l+1) \mu_x(X_1,\cdots,X_l,\pi_\ast Z_1,\cdots, \pi_\ast Z_{k-l}),\]where $\mu_x$ is in $M$, and then extend by linearity to non-decomposables. Moreover, given $\alpha\in\Omega^{k-l}(N)$, we define the corresponding position form to be $\pi^\ast\alpha$, a $(k-l)$-form on $M$. Notice that in symplectic case ($l=1$) this definition coincides with the classical momentum and position functions.
\end{definition}

\begin{definition}
Given a vector field $Y\in\Gamma(TN)$ with flow $\theta_t$, the complete lift of $Y$ is the vector field $Y^\sharp\in\Gamma(TM)$ whose flow is $(\theta_t^\ast)^{-1}$. For a decomposable multivector field $Y=Y_1\wedge\cdots\wedge Y_l\in\Gamma(\Lambda^l(TN))$ we define its complete lift $Y^\sharp$ to be $Y_1^\sharp\wedge\cdots Y_l^\sharp$ and then extend by linearity.
\end{definition}

For $\xi\in\g$, let $V_\xi$ denote its infinitesimal generator on $N$ and let $V^\sharp_\xi$ denote its infinitesimal generator on $M$. Similarly, for $p=\xi_1\wedge\cdots\wedge \xi_l$ in $\Rho_{\g,l}$ let $V_p$ denote $V_{\xi_1}\wedge\cdots\wedge V_{\xi_l}$ and $V^\sharp_p$ denote $V^\sharp_{\xi_1}\wedge\cdots\wedge V^\sharp_{\xi_l}$. Notice that by definition, we are not abusing notation by letting $V_p^\sharp$ denote both the complete lift of $V_p$ and the infinitesimal generator of $p$ under the induced action on $M$.

Lastly, note that by the equivariance of $\pi:M\to N$, we have \[\pi_\ast(V^\sharp_p)=V_p\circ\pi.\]

We now examine the bracket relations between our momentum and position forms. We first rewrite the momentum form in a different way:

\begin{proposition}
\label{momentum form}
For $Y\in\Gamma(\Lambda^l(TN))$ we have that $P(Y)=-\zeta(l+1)Y^\sharp\hk\theta$.
\end{proposition}

\begin{proof}
Let $Y=Y_1\wedge\cdots Y_l$ be an arbitrary decomposable element of $\Gamma(\Lambda^l(TN))$. Let $Z_1,\cdots, Z_{k-l}$ be arbitrary vector fields on $M$. Fix $\mu_x\in M$. Then

\begin{align*}
(Y^\sharp\hk\theta)_{\mu_x}(Z_1,\cdots, Z_{k-l})&=\theta_{\mu_x}(Y_1^\sharp,\cdots, Y_l^\sharp, Z_1,\cdots, Z_{k-l})\\
&=\mu_x(\pi_\ast Y_1^\sharp,\cdots, \pi_\ast Y_l^\sharp, \pi_\ast Z_1,\cdots, \pi_\ast Z_{k-l})\\
&=\mu_x(Y_1,\cdots,Y_l,\pi_\ast Z_1,\cdots, \pi_\ast Z_{k-l})\\
&=-\zeta(l+1)P(Y)_{\mu_x}(Z_1,\cdots, Z_{k-l}).
\end{align*}
\end{proof}
As a corollary to the above proposition, we obtain a generalization of (\ref{momentum form equal}) to multisymplectic geometry:

\begin{corollary}
For $p=\xi_1\wedge\cdots\wedge\xi_l$ in $\Rho_{\g,l}$ we have \[P(V_p)=f_l(p).\]
\end{corollary}

\begin{proof}
\del{
Let $\mu_x\in\Lambda^l(T_x^\ast N)$ be arbitrary.  Then, for arbitrary $Z_1,\cdots, Z_{k-l}\in\Gamma(TM)$, we have 

\begin{align*}
f_l(p)(\mu_x)(Z_1,\dots,Z_{k-l})&=-\zeta(l+1)(V^\sharp_p\hk\theta)_{\mu_x}(Z_1,\dots, Z_{k-l})\\
&=-\zeta(l+1)\mu_x(\pi_\ast(V^\sharp_p),\pi_\ast Z_1,\dots,\pi_\ast Z_{k-l})&\text{by definition of $\theta$}\\
&=-\zeta(l+1)\mu_x(V_{\xi_1},\dots,V_{\xi_l},\pi_\ast Z_1,\dots,\pi_\ast Z_{k-l})\\
&=P(V_p)(\mu_x)(Z_1,\dots,Z_{k-l})
\end{align*}
}This follows immediately from Proposition \ref{momentum form} since $V_p^\sharp$ is the infinitesimal generator of $p$ on $M$.
\end{proof}

\del{
\begin{remark}
If we have a group acting on $N$, then the induced action on $M$ is given precisely by the pull back. Hence, for an element $\xi\in\g$ we have that the infinitesimal generator on $M$ is exactly the complete lift of the infinitesimal generator of $\xi$ on $N$. This justifies the use of the $\sharp$ notation in the above proposition.
\end{remark}
}

In the symplectic case, given $Y\in\Gamma(TN)$ the complete lift $Y^\sharp\in\Gamma(T(T^\ast N))$ preserves the tautological forms $\theta$ and $\omega$. Hence $d(Y^\sharp\hk\theta)=Y^\sharp\hk\omega$ showing that each momentum function is Hamiltonian with Hamiltonian vector field the complete lift of the base vector field. 

In the multisymplectic case, it is no longer true that $\L_{Y^\sharp}\theta=0$ for a multivector field $Y$. Instead, we need to restrict our attention to multivector fields in the Lie kernel, which we defined in Definition \ref{Lie kernel}. We quickly recall this definition and some terminology and notation introduced in \cite{MS}. 

Any degree $l$-multivector field is a sum of multivectors of the form $Y=Y_1\wedge\cdots\wedge Y_l$. We consider the differential graded Lie algebra $(\Gamma(\Lambda^\bullet(TN)),\pd)$ where $\pd_l:\Gamma(\Lambda^l(TN))\to\Gamma(\Lambda^{l-1}(TN))$ is given by \[\pd_l(Y_1\wedge\cdots\wedge Y_l)=\sum_{1\leq i\leq j\leq l}[Y_i,Y_j]\wedge Y_1\wedge\cdots\wedge\widehat Y_i\wedge\cdots\wedge\widehat Y_j\wedge\cdots\wedge Y_l.\]As in \cite{MS}, for a differential form $\tau$, let \[(\hk \ \L)_Y\tau=\sum_{i=1}^lY_1\wedge\cdots\wedge \widehat Y_i\wedge\cdots\wedge Y_l\hk\L_{Y_i}\tau.\]

A more general version of Lemma \ref{extended Cartan} is given by Lemma 3.4 of \cite{MS}:

\begin{lemma}
\label{extended Cartan lemma}
For a differential form $\tau$ and $Y=Y_1\wedge\cdots\wedge Y_l\in\Gamma(\Lambda^l(TN))$ we have that \[Y\hk d\tau-(-1)^ld(Y\hk\tau)=(\hk \ \L)_Y\tau-\pd_l(Y)\hk\tau.\]
\end{lemma}

\begin{definition}
As in Definition \ref{Lie kernel}, we call $\Rho_{l}=\text{ker }\pd_l$ the $l$-th Lie kernel. 
\end{definition}

\begin{proposition}
\label{momentum form 2}
For an $l$-multivector field in the Lie kernel, $Y\in\Rho_{l}$, we have that $P(Y)$ is in $\Omega^{k-l}_{\mathrm{Ham}}(M)$. More precisely, $\zeta(l)Y^\sharp$ is a Hamiltonian multivector field for $P(Y)$. 
\end{proposition}
\begin{proof}
Abusing notation, let $\partial_l$ denote the differential on both $\Gamma(\Lambda^\bullet(TN))$ and $\Gamma(\Lambda^\bullet(TM))$. By definition, we have $\partial_l(Y)=0$. It follows that $\partial_l(Y^\sharp)=0$. Now, since the action on $M$ preserves $\theta$, we have that $(\hk \ \L)_{Y^\sharp}\theta=0$. Thus, by Proposition \ref{momentum form} and Lemma  \ref{extended Cartan lemma}, we have that 
\begin{align*}
d(P(Y))&=-\zeta(l+1)d(Y^\sharp\hk\theta)\\
&=-\zeta(l+1)(-1)^l(Y^\sharp\hk d\theta)\\
&=\zeta(l+1)(-1)^{l}(Y^\sharp\hk\omega)\\
&=-\zeta(l)Y^\sharp\hk\omega,
\end{align*} where in the last equality we used Remark \ref{ugly signs}.
\end{proof}

\begin{remark}
In the setup of classical Hamiltonian mechanics, the phase space of $N$ is just $T^\ast N$, and so $k=l=1$. Since $\Rho_{1}=\Gamma(TN)$ we see that we are obtaining a generalization from Hamiltonian mechanics.
\end{remark}
We now arrive at our generalization of equation (\ref{classical 1}):

\begin{proposition}
For $Y_1\in\Rho_s$ and $Y_2\in\Rho_t$ we have that
\[\{P(Y_1),P(Y_2)\}=-(-1)^{ts+s+t}P([Y_1,Y_2]) -\zeta(s+1)\zeta(t+1)d(Y_1^\sharp\hk Y_2^\sharp\hk\theta).\]

\end{proposition}

\begin{proof}
Using Proposition \ref{momentum form 2}, Remark \ref{ugly signs} and the definition of the bracket, we have
\begin{equation}\label{ugly Poisson}
\begin{aligned}\{P(Y_1),P(Y_2)\}&=(-1)^{t+1}\zeta(t)\zeta(s)Y_2^\sharp\hk Y_1^\sharp\hk\omega\\&=(-1)^{ts+t}\zeta(s+t)Y_2^\sharp\hk Y_1^\sharp\hk\omega.\end{aligned}\end{equation}On the other hand, by Proposition \ref{momentum form} and Remark \ref{ugly signs} we have \begin{align*}P([Y_1,Y_2])&=-\zeta(s+t)[Y_1^\sharp,Y_2^\sharp]\hk\theta.\end{align*} By Proposition \ref{momentum form 2} and Remark \ref{ugly signs} we have that \[d(Y_1^\sharp\hk\theta)=(-1)^{s+1}Y_1^\sharp\hk\omega\] and \[d(Y_2^\sharp\hk\theta)=(-1)^{t+1}Y_2^\sharp\hk\omega.\] Using these two equations and equation (\ref{interior equation}) we have that

\begin{align*}
[Y_1^\sharp,Y_2^\sharp]\hk\theta&=-Y_2^\sharp\hk(d(Y_1^\sharp\hk\theta))+(-1)^td(Y_2^\sharp\hk Y_1^\sharp\hk\theta)-(-1)^{st+s}Y_1^\sharp\hk Y_2^\sharp\hk\omega-(-1)^{st+s+t}Y_1^\sharp\hk(d(Y_2^\sharp\hk\theta))\\
&=(-1)^s(Y_2^\sharp\hk Y_1^\sharp\hk\omega)+(-1)^td(Y_2^\sharp\hk Y_1^\sharp\hk\theta)-(-1)^{st+s}Y_1^\sharp\hk Y_2^\sharp\hk\omega-(-1)^{st+s+t}(-1)^{t+1}Y_1^\sharp\hk Y_2^\sharp\hk\omega\\
&=(-1)^s(Y_2^\sharp\hk Y_1^\sharp\hk\omega)+(-1)^td(Y_2^\sharp\hk Y_1^\sharp\hk\theta)-(-1)^{st+s}Y_1^\sharp\hk Y_2^\sharp\hk\omega+(-1)^{st+s}Y_1^\sharp\hk Y_2^\sharp\hk\omega\\
&=(-1)^s(Y_2^\sharp\hk Y_1^\sharp\hk\omega)+(-1)^td(Y_2^\sharp\hk Y_1^\sharp\hk\theta).
\end{align*}Thus, \begin{equation}\label{ugly Poisson 2}P([Y_1^\sharp,Y_2^\sharp])=-\zeta(s+t)(-1)^s(Y_2^\sharp\hk Y_1^\sharp\hk\omega)-\zeta(s+t)(-1)^td(Y_2^\sharp\hk Y_1^\sharp\hk\theta).\end{equation}
Equating equations (\ref{ugly Poisson}) and (\ref{ugly Poisson 2}) and using Remark \ref{ugly signs} gives the result.
\end{proof}

To generalize (\ref{classical 2}) and (\ref{classical 3}) to the multisymplectic phase space, we need the following lemma:
\begin{lemma}\label{push zero}
Let $\alpha$ be an arbitrary $(k-l)$-form on $N$ and let $\pi^\ast\alpha$ be the corresponding classical position form in $\Omega^{k-l}(M)$. Then $\pi^\ast\alpha$ is Hamiltonian and $\pi_\ast(X_{\pi^\ast\alpha})=0$.
\end{lemma}

\begin{proof}
Let $q^1,\cdots,q^n$ denote coordinates on $N$, and let $\{p_{i_1\cdots i_k}; 1\leq i_1<\cdots<i_k\leq n \}$ denote the induced fibre coordinates on $M$. In these coordinates we have that \[\theta=\sum_{1\leq i_1<\cdots<i_k\leq n}p_{i_1\cdots i_k}dq^{i_1}\wedge\cdots\wedge dq^{i_k}\] so that \[\omega=-d\theta=\sum_{1\leq i_1<\cdots<i_k\leq n}-dp_{i_1\cdots i_k}\wedge dq^{i_1}\wedge\cdots\wedge dq^{i_k}.\] An arbitrary $(k-l)$-form $\alpha$ on $N$ is given by \[\alpha=\alpha_{i_1\cdots i_{k-l}}dq^{i_1}\wedge\cdots\wedge dq^{i_{k-l}}.\]Abusing notation, it follows that \[\pi^\ast\alpha=\alpha_{i_1\cdots i_{k-l}}dq^{i_1}\wedge\cdots\wedge dq^{i_{k-l}}.\]Thus,\[d\pi^\ast\alpha=\frac{\pd\alpha_{i_1\cdots i_{k-l}}}{\pd q^j}dq^j\wedge dq^{i_1}\wedge\cdots\wedge dq^{i_{k-l}}.\] An arbitrary $l$-vector field on $M$ is of the form \[X=a^{i_1\cdots i_l}\frac{\pd}{\pd q^{i_1}}\cdots\frac{\pd}{\pd q^{i_{l}}}+a^{i_1\cdots i_{l-1}}_{J_1}\frac{\pd}{\pd q^{i_1}}\cdots\frac{\pd}{\pd q^{i_{l-1}}}\wedge\frac{\pd}{\pd p^{J_1}}+\cdots+a_{J_1\cdots J_{l}}\frac{\pd}{\pd p^{J_1}}\cdots\frac{\pd}{\pd p^{J_l}}.\] Now, the multivector field $X_{\pi^\ast\alpha}$ we are looking for satisfies $X_{\pi^\ast\alpha}\hk\omega= d\pi^\ast\alpha$. An exercise in combinatorics shows that there always exists an $l$-vector field $X$ satisfying $X\hk\omega=d\pi^\ast\alpha$, proving that $\pi^\ast\alpha$ is Hamiltonian. Note we can see directly from the equality $X\hk\omega=d\pi^\ast\alpha$ that necessarily  \[a^{i_1\cdots i_l}=0.\] Thus $\pi_\ast(X_{\pi^\ast\alpha})=0$ as desired.

\del{

  \[a^{i_1\cdots i_{l-2}}_{J_1J_2}= \cdots = a_{J_1J_2\cdots J_l}=0.\] 
By taking the interior product of this vector field with $\omega$, we see that
 We can see that for this to happen, necessarily \[a^{i_1\cdots i_l}=0.\] Moreover, one can show through an exercise in combinatorics that system of equations for modelling the equation of finding $a_{J_1}^{i_1\cdots i_{l-1}}$ such that $X\hk\omega=d\pi^\ast\alpha$ always exists showing that $\pi^\ast\alpha$ is always a Hamiltonian form. Moreover, since $a^{i_1\cdots i_l}=0,$ it follows that $\pi_\ast(X_{\pi^\ast\alpha})=0.$ \del{

 Hencesatisfying the equation  To simplify the notation, let $I_{l-1}$ denote a multi-index set of length $l-1$. It follows the vector field $X_{\pi^\ast\alpha}$ we are looking for is of the from $X_{\pi^\ast\alpha}=a^{I_{l-1}}_{J_k}\frac{\pd}{\pd a^{I_{l-1}}}\wedge\frac{\pd}{\pd p^J}$

 Then $X_{\pi^\ast}\alpha$ has $a^{i_1\cdots i_{l-1}}_{J_1}=\frac{\pd\alpha_{i_1\cdots i_{k-l}}}{\pd q^j}$ and all other coefficients equal to zero. In particular then, since $a^{i_1\cdots i_l}=0$, it follows that $\pi_\ast(X_{\pi^\ast\alpha})=0$.
 }
 
 }
\end{proof}

Our generalization of (\ref{classical 2}) is:
\begin{proposition}
For $\alpha\in\Omega^{k-i}(N)$ and $\beta\in\Omega^{k-j}(N)$ we have that \[\{\pi^\ast\alpha,\pi^\ast\beta\}=0.\]
\end{proposition}

\begin{proof}
Let $Z_1,\cdots, Z_{k+1-i-j}\in\Gamma(TM)$ be arbitrary. Then,

\begin{align*}
\{\pi^\ast\alpha,\pi^\ast\beta\}(Z_1,\cdots, Z_{k+1-i-j})&=(-1)^{j+1}X_{\pi^\ast\beta}\hk X_{\pi^\ast\alpha}\hk\omega(Z_1,\cdots, Z_{k+1-i-j})\\
&=(-1)^jX_{\pi^\ast\beta}\hk(\pi^\ast d\alpha)(Z_1,\cdots, Z_{k+1-i-j})\\
&=(-1)^j\pi^\ast d\alpha(X_{\pi^\ast\beta},Z_1,\cdots, Z_{k+1-i-j})\\
&=(-1)^jd\alpha(0,\pi_\ast Z_1,\cdots, \pi_\ast X_{k+1-i-j})&\text{by Lemma \ref{push zero}}\\
&=0.
\end{align*}

\end{proof}

Our generalization of (\ref{classical 3}) is:
\begin{proposition}
For $\alpha\in\Omega^{k-i}(N)$ and $Y\in\Rho_j$, we have that \[\{\pi^\ast\alpha,P(Y)\}=-\zeta(j)\pi^\ast(Y\hk d\alpha).\]
\end{proposition}

\begin{proof}
Let $Z_1,\cdots, Z_{k+1-i-j}\in\Gamma(TM)$ be arbitrary. Then,
\begin{align*}
\{\pi^\ast\alpha,P(Y)\}(Z_1,\cdots, Z_{k+1-i-j})&=(-1)^{j+1}X_{P(Y)}\hk X_{\pi^\ast\alpha}\hk\omega(Z_1,\cdots, Z_{k+1-i-j})\\
&=(-1)^{j+1}\zeta(j)Y^\sharp\hk X_{\pi^\ast\alpha}\hk\omega(Z_1,\cdots,Z_{k+1-i-j})&\text{by Lemma \ref{momentum form 2}}\\
&=(-1)^j\zeta(j+1)Y^\sharp\hk \pi^\ast d\alpha(Z_1,\cdots,Z_{k+1-i-j})\\
&=-\zeta(j)Y^\sharp\hk \pi^\ast d\alpha(Z_1,\cdots,Z_{k+1-i-j})&\text{by Remark \ref{ugly signs}}\\
&=-\zeta(j)d\alpha(\pi_\ast Y^\sharp,\pi_\ast Z_1,\cdots, \pi_\ast Z_{k+1-i-j})\\
&=-\zeta(j)d\alpha(Y,\pi_\ast Z_1,\cdots,\pi_\ast Z_{k+1-i-j})\\
&=-\zeta(j)\pi^\ast(Y\hk d\alpha)(Z_1,\cdots, Z_{k+1-i-j}).
\end{align*}
\end{proof}
\subsection{Torsion-Free $G_2$ Manifolds}

We first recall the standard $G_2$ structure on $\R^7$.  More details for the material in this section can be found in \cite{j}.  Let $x^1,\cdots,x^7$ denote the standard coordinates on $\R^7$ and consider the three form $\varphi_0$ defined by
\[\varphi_0=dx^{123}+dx^1(dx^{45}-dx^{67})+dx^2(dx^{46}-dx^{75})-dx^3(dx^{47}-dx^{56})\] where we have ommited the wedge product signs. The stabilizer of this three form is given by the Lie group $G_2$. For an arbitrary $7$-manifold we define a $G_2$ structure to be a three form $\varphi$ which has around every point $p\in M$ local coordinates with $\varphi=\varphi_0$, at the point $p$. 

The three form induces a unique metric $g$ and volume form, vol,  determined by the equation \[(X\hk\varphi)\wedge(Y\hk\varphi)\wedge\varphi=-6g(X,Y)\text{vol}.\]From the volume form we get the Hodge star operator and hence a $4$-form $\psi:=\ast\varphi$. We will refer to the data $(M^7,\varphi,\psi,g)$ as a manifold with $G_2$ structure. We remark that the $G_2$ form $\varphi$ is more than just non-degenerate:

\begin{proposition}
The $G_2$ form $\varphi$ is fully nondegenerate. This means that $\varphi(X,Y,\cdot)$ is non-zero whenever $X$ and $Y$ are linearly independent.
\end{proposition}

\begin{proof}
See Theorem 2.2 of \cite{MS}.
\end{proof}

We will call a manifold with $G_2$ structure torsion-free if both $\varphi$ and $\psi$ are closed. A theorem of Fernandez and Gray shows that this happens precisely when $\varphi$ is parallel with respect to the induced metric $g$. Thus we see that a torsion-free $G_2$ structure is an example of a multisymplectic manifold. 

\begin{remark}
All of the results in this section will only use the fact that $\varphi$ is closed so that, in particular, all of our results holds if the $G_2$ structure is torsion-free.
\end{remark}

We now quickly recall some first order differential operators on a $G_2$ manifold, while refering the reader to section 4 of \cite{Spiro 2} for more details. Given $X\in\Gamma(TM)$ we will let $X^\flat$ denote the metric dual one form $X^\flat=X\hk g$. Conversely, given $\alpha\in\Omega^1(M)$, let $\alpha^\sharp$ denote the metric dual vector field. Recall that given $f\in C^\infty(M)$ its gradient is defined by \[\text{grad}(f)=(df)^\sharp.\] 

From the metric and the three form we can define the cross product of two vector fields.  Given $X,Y,Z\in\Gamma(TM)$ the cross product $X\times Y$ is defined by the equation \[\varphi(X,Y,Z)=g(X\times Y,Z).\] Equivalently, the cross product is defined by \[(X\times Y)=(Y\hk X\hk\varphi)^\sharp.\] In coordinates, this says that \begin{equation}\label{cross}(X\times Y)^l=X^iY^i\varphi_{ijk}g^{kl}.\end{equation}The last differential operator we will consider is the curl of a vector field. We first need to recall the following decomposition of two forms on a $G_2$ manifold. 

\begin{proposition}
The space of $2$-forms on a $G_2$ manifold has the $G_2$ irreducible decomposition \[\Omega^2(M)=\Omega^2_7(M)\oplus\Omega^2_{14}(M),\] where \[\Omega^2_7(M)=\{X\hk\varphi ; X\in\Gamma(TM)\}\] and \[\Omega^2_{14}(M)=\{\alpha\in\Omega^2(M); \psi\wedge\alpha=0\}.\] The projection maps: $\pi_7:\Omega^2(M)\to\Omega^2_7(M)$ and $\pi_{14}:\Omega^2(M)\to\Omega^2_{14}(M)$ are given by \begin{equation}\label{pi 7}\pi_7(\alpha)=\frac{\alpha-\ast(\varphi\wedge\alpha)}{3}\end{equation} and \begin{equation}\label{pi 14}\pi_{14}(\alpha)=\frac{2\alpha+\ast(\varphi\wedge\alpha)}{3}.\end{equation}
\end{proposition}

\begin{proof}
See Section 2.2 of \cite{Spiro}.
\end{proof}

We can now define the curl of a vector field. Given $X\in\Gamma(TM)$ its curl is defined by \begin{equation}\label{curl defn}(\text{curl}(X))^\flat=\ast(dX^\flat\wedge\psi).\end{equation}This is equivalent to saying that \begin{equation}\label{curl is 14}\pi_{7}(dX^\flat)=\text{curl}(X)\hk\varphi.\end{equation} In coordinates,

\begin{equation}\label{curl coord}\text{curl}(X)^l=(\nabla_aX_b)g^{ai}g^{bj}\varphi_{ijk}g^{kl},\end{equation} where $\nabla$ is the Levi-Civita connection corresponding to $g$. This is reminiscent of the fact that in $\R^3$ the curl is given by the cross product of $\nabla$ with $X$. Again, we refer the reader to Section 4.1 of \cite{Spiro 2} for more details.
\\

 We now translate our definition of Hamiltonian forms and vector fields into the language of $G_2$ geometry.  By definition, we see that a $1$-form is Hamiltonian if and only if its differential is in $\Omega^2_7(M)$. That is, \[\Omega^1_{\text{Ham}}(M)=\{\alpha\in\Omega^1(M) ; \pi_{14}(d\alpha)=0\}.\] Similarly,

\[\X_{\text{Ham}}^1(M)=\{X\in\Gamma(TM) ; X=\text{curl}(\alpha^\sharp) \text{ and } \pi_{14}(d\alpha)=0 \text{ for some } \alpha\in\Omega^1(M)\}.\]

Note that if $M$ is compact, then it follows from (\ref{pi 7}) and Hodge theory that there are no non-zero Hamiltonian $1$-forms.
\begin{proposition}
\label{Hamiltonian is curl}
If $\alpha$ is a Hamiltonian $1$-form then its corresponding Hamiltonian vector field is $\mathrm{curl}(\alpha^\sharp)$.
\end{proposition}
\begin{proof}
Since a Hamiltonian $1$-form satisfies $\pi_{14}(\alpha)=0$, this follows immediately from equation (\ref{curl is 14}). 
\end{proof}

From Proposition \ref{Hamiltonian is curl} and equation (\ref{cross}) we see that the generalized Poisson bracket is given by the cross product: \begin{equation}\label{bracket is cross} \{\alpha,\beta\}=\text{curl}(\alpha^\sharp)\times\text{curl}(\beta^\sharp),\end{equation}
for $\alpha,\beta\in\Omega^1_{\text{Ham}}(M)$.

Proposition \ref{preserve} showed that a Hamiltonian vector field preserves the $n$-plectic form. In the language of $G_2$ geometry this gives:

\begin{proposition}
Given $\alpha\in\Omega^1(M)$ with $\pi_{14}(d\alpha)=0$, the curl of $\alpha^\sharp$ preserves the $G_2$ structure. That is, \[\L_{\mathrm{curl}(\alpha^\sharp)}\varphi=0.\]
\end{proposition}

\begin{proof}
This follows immediately from Propositions \ref{Hamiltonian is curl} and \ref{preserve}.
\end{proof}

\del{

By proposition ?? of \cite{Spiro} we have that curl(curl($X$))$=.$. Hence we see that

\begin{proposition}
If $\alpha$ is Hamiltonian then so is curl($\alpha^\sharp$). In other words, if $\pi_{14}(d\alpha)=0$ then $\pi_{14}(\mathrm{curl}(\alpha^\sharp))=0$.
\end{proposition}

}
As a consequence of Proposition \ref{Poisson is Schouten}, we get the following:

\begin{proposition}
Let $\alpha$ and $\beta$ be in $\Omega^1(M)$ with $\pi_{14}(d\alpha)=0=\pi_{14}(d\beta)$. Then \[\pi_{14}(\mathrm{curl}(\alpha^\sharp)\times\mathrm{curl}(\beta^\sharp))=0.\] Moreover,
\[\mathrm{curl}(\mathrm{curl}(\alpha^\sharp)\times\mathrm{curl}(\beta^\sharp))=[\mathrm{curl}(\alpha^\sharp),\mathrm{curl}(\beta^\sharp)].\]

\end{proposition}
\begin{proof}
By equation (\ref{bracket is cross}) and Lemma \ref{Poisson is Schouten} we see that $d(\mathrm{curl}(\alpha^\sharp)\times\mathrm{curl}(\beta^\sharp))=[X_\alpha,X_\beta]\hk\omega$. Thus, $\mathrm{curl}(\alpha^\sharp)\times\mathrm{curl}(\beta^\sharp)$ is in $\Omega_7^2(M)$, showing that $\pi_{14}(\mathrm{curl}(\alpha^\sharp)\times\mathrm{curl}(\beta^\sharp))=0.$ Moreover, we have that

\begin{align*}
\text{curl}(\text{curl}(\alpha^\sharp)\times\text{curl}(\beta^\sharp))\hk\varphi&=\text{curl}(\{\alpha,\beta\})\hk\varphi&\text{by (\ref{bracket is cross})}\\
&=d(\{\alpha,\beta\})\\
&=[X_\alpha,X_\beta]\hk\varphi&\text{by Lemma \ref{Poisson is Schouten}}\\
&=[\text{curl}(\alpha^\sharp),\text{curl}(\beta^\sharp)]\hk\varphi &\text{by Proposition \ref{Hamiltonian is curl}.}
\end{align*}

The proposition now follows since $\varphi$ is non-degenerate.
\del{
By proposition \ref{Poisson is Schouten}, $X_{\{\alpha,\beta\}}=[X_\alpha,X_\beta]$ up to an element in the kernel of the $n$-plectic form. On a $G_2$ manifold the kernel of $\varphi$ is trivial. Hence, by definition 
\begin{align*}\{\alpha,\beta\}^\sharp&=X_\alpha\times X_\beta\\
&=\text{curl}(\alpha^\sharp)\times\text{curl}(\beta^\sharp)\\
\end{align*}
Since $X_{\{\alpha,\beta\}}=\text{curl}(\text{curl}(\alpha^\sharp)\times\text{curl}(\beta^\sharp))$ we obtain the result.
}
\end{proof}

We now consider the definition of a homotopy co-momentum map in the setting of a $G_2$ manifold. 
The equations defining the components of a homotopy co-momentum map, i.e. (\ref{hcmm kernel}), reduce to finding functions $f_1:\g\to\Omega^1(M)$ and $f_2:\Rho_{\g,2}\to C^\infty(M)$ satisfying 

\begin{equation}\label{G2 moment map 1} \pi_{14}(d(f_1(\xi)))=0 \text{ and curl}((f_1(\xi))^\sharp)=V_\xi. \end{equation}
\begin{equation}\label{G2 moment map 2} V_\xi\times V_\eta= -(d(f_2(\xi\wedge\eta)))^\flat.\end{equation}

We finish this section by computing a homotopy co-momentum map in the following set up, extending Example 6.7 of \cite{ms}. 

Consider $\R^7=\R\oplus\C^3$ with standard $3$-form given by \[\varphi=\frac{1}{2}\left(dz^1\wedge dz^2\wedge dz^3+d\o z^1\wedge d\o z^2\wedge d\o z^3\right)-\frac{i}{2}(dz^1\wedge d\o z^1+dz^2\wedge d\o z^2+dz^2\wedge d\o z^3)\wedge dt.\] In terms of $t,x^1,x^2,x^3,y^1,y^2,y^3$ this is \[\varphi=dx^1dx^2dx^3-dx^1dy^2dy^3-dy^1dx^2dy^3-dy^1dy^2dx^3-dtdx^1dy^1-dtdx^2dy^2-dtdx^3dy^3,\]where we have ommited the wedge signs. Equivalently, \[\varphi=\Omega_3-dt\wedge \omega_3\] where $\Omega_3$ is the standard holomorphic volume and $\omega_3$ is the standard Kahler form on $\C^3$.  That is, $\Omega_3=dz^1\wedge dz^2\wedge dz^3$ and $\omega_3=\frac{i}{2}(dz^1\wedge d\o z^1+dz^2\wedge d\o z^2+dz^3\wedge d\o z^3)$.

 As in Examples \ref{complex moment map 1} and \ref{complex moment map 2} we consider the standard action by the diagonal maximal torus $T^2\subset SU(3)$ given by $(e^{i\theta},e^{i\eta})\cdot(t,z_1,z_2,z_3)=(t,e^{i\theta}z_1,e^{i\eta}z_2,e^{-i(\theta+\eta)}z_3)$. We have $\mathfrak{t}^2=\R^2$ and that the infinitesimal generators of $(1,0)$ and $(0,1)$ are \[A=\frac{i}{2}\left(z_1\frac{\pd}{\pd z_1}-z_3\frac{\pd}{\pd z_3}-\o z_1\frac{\pd}{\pd\o z_1}+\o z_3\frac{\pd}{\pd \o z_3}\right)\]and \[B=\frac{i}{2}\left(z_2\frac{\pd}{\pd z_21}-z_3\frac{\pd}{\pd z_3}-\o z_2\frac{\pd}{\pd\o z_2}+\o z_3\frac{\pd}{\pd \o z_3}\right)\] respectively.
 
By Example \ref{complex moment map 1} it follows that 
\begin{align*}
A\hk\varphi&=A\hk(\Omega_3-dt\wedge\omega_3)\\
&=\frac{1}{2}d(\text{Im}z_1z_3dz^2)-\frac{1}{4}dt\wedge d(|z_1|^2-|z_3|^2)\\
&=\frac{1}{2}d\left(\text{Im}(z_1z_3dz^2)-\frac{1}{2}(|z_1|^2-|z_3|^2)dt\right).
\end{align*}
Similarly,
\begin{align*}
B\hk\varphi=\frac{1}{2}d\left(\text{Im}(z_1z_2dz^3)-\frac{1}{2}(|z_1|^2-|z_2|^2)dt\right).
\end{align*}
It follows that \[f_1((1,0))=\frac{1}{2}\text{Im}(z_1z_3dz^2)-\frac{1}{4}(|z_1|^2-|z_3|^2)dt\] and \[f_1((0,1))=\frac{1}{2}\text{Im}(z_1z_2dz^3)-\frac{1}{4}(|z_1|^2-|z_2|^2)dt\] give the first component of a homotopy co-momentum map. Plugging in $f_1((1,0))$ and $f_1((0,1))$ into (\ref{pi 14}) shows that \[\pi_{14}(f_1((1,0)))=0=\pi_{14}(f_1((0,1))).\] Moreover, using (\ref{curl coord}) one can directly verify that \[\text{curl}(f_1((1,0)))^\sharp=A\] and \[\text{curl}(f_1((0,1)))^\sharp=B,\] confirming (\ref{G2 moment map 1}).

Using Example \ref{complex moment map 2} it follows that
\begin{align*}
B\hk A\hk\varphi&=B\hk A\hk(\Omega_3-dt\wedge \omega_3)\\
&=B\hk A\hk\Omega_3\\
&=\frac{1}{4}d(\text{Re}(z_1z_2z_3)).
\end{align*}

Thus the second component of the homotopy co-momentum map is given by  \[f_2(A\wedge B)=\frac{1}{4}\text{Re}(z_1z_2z_3),\]in accordance with Example 6.7 of \cite{ms}.
Written out in the coordinates $t$, $x^1$, $x^2$, $x^3$, $y^1$, $y^2$, $y^3,$ the infinitesimal vector fields coming from the torus action are 
\[A=\frac{1}{2}\left(-y^1\frac{\pd}{\pd x^1}+y^3\frac{\pd}{\pd x^3}+x^1\frac{\pd}{\pd y^1}-x^3\frac{\pd}{\pd y^3}\right),\]
\[B=\frac{1}{2}\left(-y^2\frac{\pd}{\pd x^2}+y^3\frac{\pd}{\pd x^3}+x^2\frac{\pd}{\pd y^2}-x^3\frac{\pd}{\pd y^3}\right).\]

Using the metric to identify $1$-forms and vector fields, equation  (\ref{cross}) gives the cross product of $A$ and $B$ to be 
\begin{align*}4(A\times B )&=(y^2y^3-x^2x^3)\frac{\pd}{\pd x^1}+(y^1y^3-x^1x^3)\frac{\pd}{\pd x^2}+(y^1y^2-x^1x^2)\frac{\pd}{\pd x^3}+\\
&+(x^2y^3+x^3y^2)\frac{\pd}{\pd y^1}+(x^3y^1+x^1y^3)\frac{\pd}{\pd y^2}+(x^1y^2+x^2y^1)\frac{\pd}{\pd y^3}\\
&=d(x^1x^2x^3-x^1y^2y^3-y^1x^2y^3-y^1y^2x^3)\\
&=d(\text{Re}(z_1z_2z_3))
\end{align*}
confirming equation (\ref{G2 moment map 2}).  We thus have extended Example 6.7 of \cite{ms} by obtaining a full homotopy co-momentum map for the diagonal torus action on $\R^7$ with the standard torsion-free $G_2$ structure.

\section{Concluding Remarks}

This work poses many natural questions for future research. The following are just a few ideas:

\begin{itemize}
\item The existence and uniqueness of homotopy co-momentum maps has been studied in \cite{existence 1} and \cite{existence 2}, for example. However, in this paper we were mostly concerned with co-momentum maps restricted to the Lie kernel, i.e. weak co-momentum maps. It would thus be desirable to have results on the existence and uniqueness of these restricted maps.

It is clear that a co-momentum map restricts to a co-momentum map on the Lie kernel. That is, if a collection of maps satisfies equation (\ref{hcmm}), then it satisfies equation (\ref{hcmm kernel}). However, are there examples of weak co-momentum maps which do not come from the restriction of a full co-momentum map?  As mentioned throughout the paper, this question is currently being investigated in \cite{future}.

\item In our work, we provided a few examples of multi-Hamiltonian systems. What are some examples of other physical or interesting multi-Hamiltonian systems to which this work could be applied?

\item In Section 6.1 we generalized the classical momentum and position functions on the phase space of a manifold to momentum and position forms on the multisymplectic phase space. Since, as discussed in \cite{Marsden}, the classical momentum and position functions play an important role in connecting classical and quantum mechanics, a natural question is if there is an analogous application of our more general theory to quantum mechanics?
\end{itemize}

\del{where in the last equality we have used the standard metric to identify vector fields and $1$-forms. Thus, by (\ref{G2 moment map 2}), we see that we should define the second component of the homotopy co momentum map to be \[f_2(A\wedge B)=x^1x^2x^3-x^1y^2y^3-y^1x^2y^3-y^1y^2x^3.\] A straightforward computation shows that \[x^1x^2x^3-x^1y^2y^3-y^1x^2y^3-y^1y^2x^3=-\frac{1}{2}\text{Re}(z_1z_2z_3)\] so that an alternative way of writing this component of the moment map is \[f_2(A\wedge B)=-\frac{1}{2}\text{Re}(z_1z_2z_3).\] This agrees with the computation done in example 6.7 of \cite{ms}. 
}
\del
{

We now try to find the first component of the homotopy co-momentum map. 

Taking the interior product of $\varphi$ by $B$ gives

\begin{align*}
\frac{2}{i}(B\hk\varphi)&=z_2dz^3dz^1+\frac{i}{2}z_2d\o z^2\wedge dt-z_3dz^1\wedge dz^2-\frac{i}{2}z_3d\o z^3\wedge dt-\\
&-\o z_2d\o z^3\wedge d\o z^1+\frac{i}{2}\o z_2dz_2\wedge dt+\o z_3d\o z^1\wedge d\o z^1-\frac{i}{2}\o z_3dz_3\wedge dt\\
&=-d\left(\frac{1}{4}(|z_2|^2-|z_3|^2)dt-\text{Im}(z_1z_3dz^2)\right)
\end{align*}

It follows that \[f_1(B)=-\frac{1}{4}(|z_2|^2-|z_3|^2)dt-\text{Im}(z_1z_3dz^2).\] Similarly, \[f_1(A)=-\frac{1}{4}(|z_1|^2-|z_3|^2)dt-\text{Im}(z_2z_3dz^1).\]
}
\del{

$A\hk\varphi$
\[\frac{i}{2}\left[z_2dz_3\wedge dz_1-\frac{i}{2}z_2d\o z^2\wedge dt - z_3dz^1\wedge dz^2-\frac{i}{2}z_3d\o z^3dt-\o z_2d\o z^3\wedge d\o z^1+\frac{i}{2}\o z_2dz_2\wedge dt+\o z_3d\o z^1\wedge d\o z^2-\frac{i}{2}\o z_3dz_3\wedge dt\right]\]

A computation then shows that \[A\hk\varphi=-d\left(\frac{1}{4}\left(|z_2|^2-|z_3|^2\right)dt-\text{Im}(z_1z_3dz^2)\right)\]and \[B\hk\varphi=-\left(\frac{1}{4}\left(|z_1|^2-|z_3|^2\right)dt-\text{Im}(z_2z_3dz^1)\right).\]Moreover,  \[B\hk A\hk\varphi=-\frac{1}{2}d(\text{Re}(z_1z_2z_3)).\]

We see that  \[f_1(A)=-\frac{1}{4}\left(|z_2|^2-|z_3|^2\right)dt-\text{Im}(z_1z_3dz^2),\]

\[f_1(B)=-\frac{1}{4}\left(|z_1|^2-|z_3|^2\right)dt-\text{Im}(z_2z_3dz^1),\] and \[f_2(A\wedge B)=-\frac{1}{2}\text{Re}(z_1z_2z_3).\]

We also can see directly that \[d(f_1(A))=A\hk\varphi \text{ and } d(f_1(B))=B\hk\varphi\] and \[A\hk B\hk\varphi=d(f_2(A\wedge B))\] as desired.

}

\del{

\subsection{Symmetries and Conserved Quantities on the Multisymplectic Tangent Bundle}

Given a multisymplectic manifold $(M,\omega)$ we can lift $\omega$ to a multisymplectic form on $TM$. I believe by lifting, we can take a moment map on $M$ to a moment map on $TM$ and symmetries and conserved quantities can also be lifted.

If the form is the volume form then each of the spaces $\Lambda^k(TM)$ have a canonical $n$-plectic structure, the complete lift, namely the pull back of the form on $\Lambda^k(T^\ast M)$. Hence in this case, all symmetries are Lagrangian submanifolds of $T(\Lambda^k(TM))$. 

From $(M,\omega)$ get $(TM,\omega^c)$...Can you take vertical (complete lift) of arbitrary vector field.. you can take vertical lift of any form. Then vertical lift of symmetry is symmetry.

\subsection{Symmetries and Conserved Quantities From Isotropy Subgroups} 
Fix a multi-Hamiltonian system $(M,\omega,H)$.  Let $G$ act on the manifold and fix $p\in\Rho_{\g,k}$. Note that by proposition \ref{Schouten is closed} the adjoint action on $\Lambda^k\g$ preserves $\Rho_{\g,k}$. Let $G_p$ denote the isotropy subgroup for the adjoint action of $G$ on $\Rho_{\g,k}$ and let $\g_p$ denote its Lie algebra. That is, $\g_p=\{\xi\in\g: [\xi,p]=0\}$. Let $G_p^0$ denote the connected component of the identity in $G_p$.

\begin{proposition}
We have that $\xi\wedge p$ is in $\Rho_{\g,k}$ if and only if $\xi$ is in $\g_p$.
\end{proposition}
\begin{proof}
This follows from proposition \ref{Schouten is closed}.
\end{proof}

We consider the multi-Hamiltonian system $(M,V_p\hk\omega, V_p\hk H)$. We summarize the result of section 3.2 of \cite{cq} in the following proposition.

\begin{proposition}
The form $V_p\hk\omega$ is a closed $(n+1-k)$-form on $M$ that is invariant under the action of $G_p$.  Moreover, $V_p\hk H$ is $G_p$ invariant and it is a a Hamiltonian form in $(M,V_p\hk \omega)$ with Hamiltonian vector field $X_H$.

\end{proposition}

Furthermore, in section 3.2 of \cite{cq}, it is shown how a co-momentum map $(f):\Lambda^\bullet\g\to L_\infty(M,\omega)$ induces a co-momentum map for $(M,V_p\hk\omega,V_p\hk X)$. In proposition 3.8 of \cite{cq} it is proven that 

\begin{proposition}
A co-momentum map for the $G_p^0$ action on $(M,V_p\hk\omega)$ is given by $(f^p):\Lambda^\bullet\g_p\to L_\infty(M,V_p\hk\omega)$ with components \[f_j^p:\Lambda^j\g_p\to \Omega^{n-k-j}(M) \ \ \ \ \ \ \ \ \ \ q\mapsto-(-1)^{k(k+1)}f_{j+k}(q\wedge p)\]
\end{proposition}

We now apply Noether's theorem to this setup.

We see that if $q$ is in $\Rho_{\g_p,l}$ then we get a conserved quantity in $(M,V_p\hk\omega,V_p\hk H)$. This symmetry is going to be an element of $\Gamma(\Lambda^{k+j-1}(TM))$ since $f_j(q)$ is in $\Omega^{n+1-k-j}(M)$. We claim that a symmetry is given by $V_q\wedge V_p$ and we prove this as a corollary to something more general.

\begin{proposition}
Let $Y\in\Gamma(\Lambda^l(TM))$ be a local symmetry in $(M,\omega,H)$. Then $Y\wedge V_p$ is a local symmetry in $(M,V_p\hk\omega, V_p\hk H)$. It is also a local symmetry in $(M,\omega,H)$.
\end{proposition}

\begin{proof}
By we need to show that 

\[[Y\wedge V_p,X_H]\hk(V_p\hk \omega)=0\]

We have that $[Y\wedge V_p, X_H]=Y\wedge[V_p,X_H] + [Y,X_H]\wedge V_p$ Hence

\begin{align*}
[Y\wedge V_p,X_H]\hk(V_p\hk \omega)&=Y\wedge[V_p,X_H]\hk(V_p\hk\omega) + [Y,X_H]\wedge V_p(V_p\hk\omega)\\
&=Y\hk V_p\hk[V_p\hk X_H]\hk\omega +V_p\hk V_p\hk[Y,X_H]\hk\omega\\
&=Y\hk V_p\hk[V_p\hk X_H]\hk\omega\\
&=0
\end{align*}
\end{proof}
This confirms the above proposition. We know that $V_p$ is a symmetry, and the infintesimal generator of $f^p_j(q)=f_{j+k}(q\wedge p)$ is $V_q\wedge V_p$.

\subsection{Multisymplectic Legendre Transform}

In \cite{Marsden} a statement of Noether's is given as Corollary 4.2.14. It states that

We demonstrate how this statement can be extended to multisymplectic geometry. In particular, we relate Hamiltonian system given in example ?? $(T^\ast M,G, H)$ to the induced multi-Hamiltonian system $(TM,\mathbb{FL}^\ast\theta_L)$.

\subsection{Transgression of Continuous Symmetries}

Let $M$ be a manifold and let $\Sigma$ be a compact, oriented manifold without boundary. Let $M^\Sigma=C^\infty(\Sigma,M)$, the space of smooth maps from $\Sigma$ to $M$. Given $V\in\Gamma(TM)$, get $V^l\in\Gamma(TM^\Sigma)$, where $V^l$ is the image under the transgression map. A multisymplectic structure on $M$ gives a multisymplectic structure on $M^\Sigma$, by the transgression map.  The transgression map also takes a moment map on $M$ to one on $M^\Sigma$. It's true that a transgression map sends a symmetry to a symmetry and preserves Noether's theorem.

}

\end{document}